\documentclass[a4paper]{amsart}
\usepackage{amsmath}
\usepackage{enumerate}
\usepackage{bbm}
\usepackage{todonotes}
\usepackage[colorlinks,pdfdisplaydoctitle]{hyperref}

\makeatletter
\providecommand\@dotsep{5}
\makeatother

\newtheorem{proposition}{Proposition}
\newtheorem{corollary}[proposition]{Corollary}
\newtheorem{lemma}[proposition]{Lemma}
\newtheorem{theorem}{Theorem}
\theoremstyle{definition}
\newtheorem{definition}{Definition}
\newtheorem{remark}[proposition]{Remark}

\newcommand\strt{\protect{\,\circ}}

\DeclareMathOperator\re{Re}
\DeclareMathOperator\im{Im}

\begin{document}
\author{David Barbato, Francesco Morandin}
\title[Stochastic inviscid shell models]{Stochastic inviscid shell models:\\ well-posedness and anomalous dissipation.}
\begin{abstract}
In this paper we study a stochastic version of an inviscid shell model
of turbulence with multiplicative noise. The deterministic counterpart
of this model is quite general and includes inviscid GOY and Sabra
shell models of turbulence. We prove global weak existence and
uniqueness of solutions for any finite energy initial
condition. Moreover energy dissipation of the system is proved in
spite of its formal energy conservation.
\end{abstract}

\maketitle

\section*{Introduction}
In recent years shell models of turbulence have attracted interest for
their ability to capture some of the statistical properties and
features of three-dimensional turbulence, while presenting a structure
much simpler than Navier-Stokes and Euler equations.

The main idea behind shell models is to summarize in a unique variable
$u_n$ (usually complex-valued) all the modes with wave number $k$
inside the shell $\lambda^n<|k|<\lambda^{n+1}$. Just like
Navier-Stokes equations written in Fourier coordinates, the functions
$\{u_n\}_{n\in\mathbb N}$ satisfy an infinite system of coupled
ordinary equations, where the non-linear term is quadratic and
formally preserves energy.

Shell models are however drastic modifications of Navier-Stokes
equations. Firstly the variables $\{u_n\}_{n\in\mathbb N}$
representing three-dimensional shells (logarithmically equispaced) are
one-dimensionally indexed by $\mathbb N$. Secondly the shells are
allowed to interact only locally.  The choice to allow only
finite-range interactions is a crucial simplification both from
analytical and numerical perspective but it is well justified inside
the Kolmogorov theory of homogeneous turbulence, where one neglects
energy exchanges between modes whose wave numbers differ for more than
one order of magnitude.

These two characteristics of the shell models represent both their
weakness and their strength. The main weaknesses are the loss of the
geometry and the restriction to questions concerning the turbulence
energy cascade only. The strengths are several: from a numerical
perspective, the lower number of degrees of freedom allows for more
accurate simulations at high Reynolds numbers (although the
implementation of these simulations is not easy); from an analytical
perspective, the simpler structure of the problem leads to sharper
results both for the well-posedness and the understanding of the
anomalous scaling exponents.

A review of the subject that focuses in particular on these aspects
can be found in Biferale~\cite{Biferale2003} which is devoted to the
turbulence energy cascade and collects results concerning the
structure function $S_p(k_n)=E[|u_n|^p]$ together with numerical
evidence and analytical conjectures about anomalous exponents.

\subsection*{Main shell models}
There are several different shell models in literature. The most
studied are the GOY, introduced in Gledzer~\cite{Gledzer1973} and
Ohkitani and Yamada~\cite{OhkYam1987} and the Sabra introduced in
L'vov \emph{et al.}~\cite{sabra1998}. Then there are two models with
interactions that are somewhat simpler to study: one was introduced in
Obukhov~\cite{obukhov1971} and the other in Desnianskii and
Novikov~\cite{DesNov74} and in Katz and Pavlovi\'c~\cite{KatPav04}.
While all of the previous have variables indexed by $\mathbb N$, there
are also generalizations where the set of indexes is a regular tree
(e.g.~again in the first part of Katz and Pavlovi\'c~\cite{KatPav04}
and in Barbato \emph{et al.}~\cite{BarBiaFlaMor2013}) which is closer
to a true wavelet formulation of Navier-Stokes.

For the viscous versions of GOY and Sabra, well-posedness, global
regularity of solutions and smooth dependence on the initial data are
known (Constantin, Levant and Titi~\cite{ConLevTit2006}). On the other
hand, for the inviscid case less is known, the state of the art being
Constantin, Levant and Titi~\cite{ConLevTit2007} where the authors
prove global existence of weak solutions and, for sufficiently smooth
initial conditions, uniqueness and regularity for small times.

For the simpler shell models there are stronger results in the viscous
case (see among the others Barbato, Morandin and
Romito~\cite{BarMorRom2011} and Cheskidov and
Friedlander~\cite{CheFri09}), and even the inviscid case is understood
quite well (the main results can be found in Katz and
Pavlovi\'c~\cite{KatPav04}, Kiselev and Zlato\v s~\cite{KisZla05},
Cheskidov, Friedlander and Pavlovi\'c~\cite{CheFriPav2007}
and~\cite{CheFriPav2010}, Barbato, Flandoli and
Morandin~\cite{BarFlaMor2010CRAS} and \cite{BarFlaMor2011TAMS} and
Barbato and Morandin~\cite{BarMor2013nodea}).

Recently some stochastic shell models have been also proposed.  An
additive-noise version of the viscous GOY which is globally well-posed
was introduced in Barbato \emph{et al.}~\cite{BarBarBesFla2006}. The
existence of invariant measures was proved in Bessaih and
Ferrario~\cite{BesFer2012}. In Manna, Sritharan and
Sundar~\cite{ManSriSun2009} and Chueshov and Millet~\cite{ChuMil2009}
a small multiplicative noise version of the GOY model is studied;
well-posedness and a large deviation principle are established.

Finally, in Barbato, Flandoli and Morandin~\cite{BarFlaMor2010PAMS}
and~\cite{BarFlaMor2011AAP} a stochastic version of the inviscid
Novikov model was proposed, which is then generalized to the
tree-indexed Novikov model in Bianchi~\cite{Bianchi2013}. In these
last models the noise term is multiplicative, and it is tailored to be
formally energy-preserving. The cited papers prove global
well-posedness of weak solutions for both models and anomalous
dissipation for the former. (By anomalous dissipation we denote the
property by which the total energy of the system decreases in spite of
the formal conservativity of the dynamics.)

This type of noise is both elegant from an analytical point of view
and physically meaningful in the sense that the interactions of Euler
equations neglected in the shell models can be thought to be some sort
of residual term which would behave (statistically) in a similar way.

\subsection*{Main results of the paper}
In this paper we study a general stochastic inviscid shell model, with
a multiplicative noise term similar to the one in Barbato, Flandoli
and Morandin~\cite{BarFlaMor2010PAMS}. We restrict ourselves to
indexes in $\mathbb N$, but we allow the variables $X_n$ to be
multidimensional.

This very general system of equations, given in~\eqref{e:mainmodel},
includes a stochastic version~\eqref{e:goy_stochastic} of the inviscid
GOY model~\eqref{e:goy} and a stochastic
version~\eqref{e:sabra_stochastic} of the inviscid Sabra
model~\eqref{e:sabra}.

The noise term is formally conservative, in the sense that it acts
only on the transport of energy, without giving or taking any part of
it. One general way to obtain such a noise is the following. Suppose
that some shell model satisfies $\frac{d}{dt}u=B(u,u)$ with $Re\langle
u,B(u,v)\rangle=0$, then also the system $du=B(u,u dt+\sigma \circ
dW)$ will be formally conservative, the term $\sigma \circ dW$ being a
perturbation acting only on the transport.

The first aim of this paper is to prove that for the general
model~\eqref{e:mainmodel} there are global weak existence and
uniqueness in law of $l^2$ solutions. This result, to the knowledge of
the authors, is the first result of global well-posedness for the
inviscid GOY and Sabra models, both deterministic and stochastic.

The second important result concerns anomalous dissipation. Theorems 7
and 8 state that for both stochastic inviscid GOY and Sabra models
energy decreases with positive probability at all times, that it
becomes arbitrarily small again with positive probability, and that if
the initial energy is small enough, the solution converges to zero at
least exponentially fast almost surely and in $L^2$.

\subsection*{Strategy and organization of the paper}

Section~\ref{sec:1-model} introduces and describes the general
stochastic inviscid shell model~\eqref{e:mainmodel}. This is a subtle
matter, since the requirement that the noise acts on the transport
term in a conservative way leads to several algebraic conditions.

One of the key ideas of the paper is to use Girsanov theorem to study
the original problem through an auxiliary linear
system. Sections~\ref{sec:2-auxiliary_eq} and~\ref{sec:3-girsanov} are
devoted to establish the relation between the linear and non-linear
systems.

In Section~\ref{sec:4-uniqueness} from the linear system we deduce the
evolution equations for the second moment of components, which turn
out to form a deterministic linear system and can be conveniently
studied through the theory of q-matrices of continuous-time Markov
chains. The first consequence is uniqueness of solutions: strong for
the auxiliary linear system (Theorem~\ref{thm:uniqueness_linear}), and
in law for the main model (Theorem~\ref{thm:uniqueness_nonlinear}).

Existence of global solutions is classical and straightforward, and is
detailed in Section~\ref{sec:5-existence}. This concludes
well-posedness for the main model.

Section~\ref{sec:6-dissipation} is devoted to anomalous dissipation,
which is deduced from the behaviour of the continuous-time Markov
chain associated to the equations for the second moment of
components. In particular the chain turns out to be dishonest, in the
sense that a.s.\ it reaches infinity in finite time. This ``loss of
mass'' pulled back to the initial system becomes a loss of energy
towards higher and higher components, which is what we call anomalous
dissipation.

To formalize the link between the chain and the main model, one needs
two steps: Borel-Cantelli lemma to get a.s.\ statements about the
auxiliary linear system, and Novikov condition for Girsanov theorem to
pass from the auxiliary system to the main model.
Theorem~\ref{thm:expo_decr_nonlinear} shows that if the initial energy
is small with respect to the noise, Novikov condition holds also at
$t=\infty$, and so even the exponential decay of energy can be deduced.

Finally we need to show that indeed GOY and Sabra models can be
included in the general model~\eqref{e:mainmodel} and summarize the
results for these two models. This is done straightforwardly in
Section~\ref{sec:7-applications}.

\section{Main model and formal requirements}\label{sec:1-model}

The general model of this paper is equation~\eqref{e:mainmodel}
below. Since it is both complex and written in a synthetic but
unfamiliar way, it will be helpful to start with a particular example
to be kept in mind as a reference.

Consider on a complete filtered probability space $(\Omega, \mathcal
F_\infty, (\mathcal F_t)_{t\geq0}, P)$ the infinite system of
stochastic differential equations in Stratonovich form below
\begin{multline*}
dX_n=\lambda^{n-1}X_{n-1}^2dt-\lambda^nX_nX_{n+1}dt\\
+\lambda^{n-1}X_{n-1}\strt dW_{n-1}-\lambda^nX_{n+1}\strt dW_n
,\qquad n\geq1,
\end{multline*}
where $(W_n)_{n\geq0}$ is a sequence of independent Brownian motions
and $X_0\equiv0$. This model is a stochastic version of the inviscid
Novikov shell model, and was introduced in Barbato, Flandoli and
Morandin~\cite{BarFlaMor2010PAMS} and~\cite{BarFlaMor2011AAP}. The two
deterministic terms are coupled in such a way to cancel when we sum
$X_ndX_n$ over $n$. Apart form this, they are of the same form,
representing an interaction between $X_n$ and the product of two other
components. The two stochastic terms are coupled among themselves in
the same way, and moreover each of them is associated to one of the
deterministic terms, so that the equation rewrites
\[
dX_n=\lambda^{n-1}X_{n-1}(X_{n-1}dt+\strt dW_{n-1})
-\lambda^nX_{n+1}(X_ndt+\strt dW_n)
,\qquad n\geq1.
\]
We want to generalize this model to different types of interactions
and multidimensional structure. We will also try to keep the notation
as less cumbersome as possible, and in this view we will rewrite this
as a sum over a set of interaction terms $I$, that will include pairs
of cancelling interactions.

We are finally able to write the general model.
\begin{equation}\label{e:mainmodel}
dX_n=\sum_{i\in I}k_{i,n}B_i(X_{n+r_i},X_{n+h_i}dt+\sigma\strt dW_{i,n+h_i}),\qquad n\geq1.
\end{equation}
Here each $X_n$ is a $d$-dimesional real-valued stochastic process,
$I$ is some finite set with an even number of elements and
$(W_{i,n})_{i\in I, n\in\mathbb Z}$ is a family of $d$-dimensional
brownian motions (independent apart from some deterministic relations
explained below).  For all $n\in\mathbb Z$ and $i\in I$, $k_{i,n}$ is
a real constant, $B_i$ is a bilinear operator on $\mathbb R^d$ while
$r_i$ and $h_i$ are integer numbers.

In the example of Novikov model given above, $I$ has two elements 1
and 2, $d=1$, $B_i(a,b)=ab$ for $i=1,2$, the coefficients are given by
\[
\begin{array}{c|ccc}
i & r_i & h_i & k_{i,n} \\
\hline
1 & -1 & -1 & \lambda^{n-1} \\
2 & \phantom-1 & \phantom-0 & -\lambda^n
\end{array}
\]
and the Brownian motions are independent apart from $W_{1,n}=W_{2,n}$
a.s.\ for all $n$.

Going back to the general model, since $r_i$ and $h_i$ may be
negative, we pose $X_n=0$ for $n\leq0$ and $k_{i,n}=0$ for $i,n$ such
that $n+r_i\leq0$ or $n+h_i\leq0$. We will also require that $\overline
h:=\max_ih_i\geq0$ otherwise $X_n$ is constant for $n\leq -\overline h$.

We now list a first set of basic requirements on these models
\begin{enumerate}[\hspace{2em}\it i.]
\item \emph{Finite range:} $I$ is a finite set.
\item \emph{No self interactions:} $r_i\neq0$ for all $i\in I$.
\item \emph{Exponential coefficients:} $k_{i,n}=\lambda^nk_i$ for all
  $i\in I_n$ and $n\geq1$; here $\lambda>1$ and $k_i$ are real numbers
  and $I_n:=\{i\in I:n+r_i\geq1, n+h_i\geq1\}$. If $i\notin I_n$
  then $k_{i,n}=0$.
\end{enumerate}
The fourth but very important requirement is the \emph{formal} (also
called local) conservation of energy, which is assured by some
cancellations, as described below. The intuitive meaning of the
conditions detailed below is that $I$ must be formed by pairs
$\{i,\tilde\imath\}$ of cancelling interactions such that for all $n$
there exists $\tilde n=n+r_i$ such that
\begin{multline*}
k_{i,n}\left\langle X_n, B_i(X_{n+r_i},X_{n+h_i}dt+\sigma\strt dW_{i,n+h_i})\right\rangle\\
+k_{\tilde\imath,\tilde n}\left\langle X_{\tilde n}, B_{\tilde\imath}(X_{\tilde n+r_{\tilde\imath}},X_{\tilde n+h_{\tilde\imath}}dt+\sigma\strt dW_{\tilde\imath,\tilde n+h_{\tilde\imath}})\right\rangle=0.
\end{multline*}
Here $\langle\cdot,\cdot\rangle$ denotes the scalar product in
$\mathbb R^d$.

To make things formal and clean we need to introduce this definition.
\begin{definition}
Suppose $\tau$ is a permutation of $I$ with no fixed point and such
that $\tau=\tau^{-1}$. Let $I^*$ be any subset of $I$ such that $I$ is
the disjoint union of $I^*$ and $\tau(I^*)$.  We say that a family
of $d$-dimensional Brownian motions $W=(W_{i,n})_{i\in I, n\in\mathbb
  Z}$ is symmetric with respect to $\tau$ if the restriction of $W$ to
$I^*\times\mathbb Z$ is a family of independent Brownian motions and
$W_{\tau(i),n}=W_{i,n}$ a.s.~for all $i\in I$ and $n\in\mathbb Z$.
\end{definition}
Clearly this definition does not depend on the particular choice of
$I^*$, but nevertheless by invoking this definition, we will
implicitly suppose that we are fixing the set $I^*$.

We can now state the fourth requirement. We will use the notation
$\tilde\imath = \tau(i)$.
\begin{enumerate}[\hspace{2em}\it i.]\setcounter{enumi}{3}
\item \emph{Local conservativity:} there exists $\tau$ such that $W$
  is symmetric with respect to $\tau$ and the following relations
  hold for all $i\in I$
\begin{align}
k_{\tilde\imath}&=-k_i\lambda^{-r_i},\label{e:cancellations_k}\\
\langle u,B_{\tilde\imath}(v,w)\rangle&=\langle v,B_i(u,w)\rangle,\qquad\forall u,v,w\in\mathbb R^d,\label{e:alias_B}\\
r_{\tilde\imath}&=-r_i,\label{e:cancellations_r}\\
h_{\tilde\imath}&=h_i-r_i,\label{e:alias_h}
\end{align}
\end{enumerate}

\begin{remark}
These algebraic requirements are meaningful, in the sense that given
$I$ and $\tau$, there exist $k$, $B$, $r$ and $h$ satisfying
them. Truly, it is easy to verify that however we define these objects
on $I^*$, there is exactly one extension on all $I$ satisfying the
above conditions.
\end{remark}

The following lemma summarizes some other trivial but useful
consequences of the requirements above.

\begin{lemma}\label{le:some_alg_conseq}
Let $\varphi$ be the automorphism on $I\times\mathbb Z$ defined by
$\varphi(i,n):=(\tilde\imath,\tilde n):=(\tau(i),n+r_i)$. Then there
exists $\Delta\subset I\times\mathbb Z$ such that
$\varphi(\Delta)=\Delta^c$. Moreover  the following relations
  hold for all $i\in I$ and $n\in\mathbb Z$
\begin{align*}
k_{\tilde\imath,\tilde n}&=-k_{i,n}\\
\tilde n&=n+r_i\\
n&=\tilde n+r_{\tilde\imath}\\
W_{\tilde\imath,\tilde n+h_{\tilde\imath}}&=W_{i,n+h_i}\quad\text{a.s.}
\end{align*}
\end{lemma}
In particular it is now straightforward that for all $i$ and $n$,
\begin{multline}\label{e:cancellations}
k_{i,n}\left\langle X_n, B_i(X_{n+r_i},X_{n+h_i}dt+\sigma\strt dW_{i,n+h_i})\right\rangle\\
+k_{\tilde\imath,\tilde n}\left\langle X_{\tilde n}, B_{\tilde\imath}(X_{\tilde n+r_{\tilde\imath}},X_{\tilde n+h_{\tilde\imath}}dt+\sigma\strt dW_{\tilde\imath,\tilde n+h_{\tilde\imath}})\right\rangle=0
\end{multline}

Since by the Stratonovich form of It\^o formula we have
\begin{equation*}
d\langle X_n,X_n\rangle
=2\langle X_n,\strt dX_n\rangle,
\end{equation*}
if we sum formally these quantities over $n$ substituting
\eqref{e:mainmodel} and using \eqref{e:cancellations}, we have
$\sum_{n\geq1}d|X_n|^2=0$, so we may expect these models to be
conservative.

Actually, this is in general not true. Rigorous arguments in the
following sections will show that
$d\sum_{n\geq1}|X_n|^2\neq\sum_{n\geq1}d|X_n|^2$ and that
$\sum_{n\geq1}|X_n|^2$ decreases with positive probability.

\section{\texorpdfstring{It\=o}{Ito} formulation and auxiliary equation}\label{sec:2-auxiliary_eq}

We prefer to reformulate equation~\eqref{e:mainmodel} with It\=o
integration. Proposition~\ref{prop:equivalencestratito} below states
that the equivalent It\=o differential equations are the following
\begin{equation}\label{e:itodiffmodel}
dX_n=\sum_{i\in I}k_{i,n}B_i(X_{n+r_i},X_{n+h_i}dt+\sigma dW_{i,n+h_i})-\frac{\sigma^2}2\sum_{i\in I}k_{i,n}^2L_iX_ndt
,\qquad n\geq1
\end{equation}
where $L_i$ is the linear map on $\mathbb R^d$ given by
$L_i:=B_iB_i^T$. (Here $B_i$ is interpreted as a linear map from
$\mathbb R^{d^2}$ to $\mathbb R^d$.) In components,
$L_i^{\alpha,\beta}=\sum_{\gamma,\delta}B_i^{\alpha,\gamma,\delta}B_i^{\beta,\gamma,\delta}$.

We will also introduce the auxiliary \emph{linear} system of equations
and their solutions. This will be needed afterwards, for Girsanov
theorem.
\begin{equation}\label{e:itodifflinearmodel}
dX_n=\sum_{i\in I}k_{i,n}B_i(X_{n+r_i},\sigma dW_{i,n+h_i})-\frac{\sigma^2}2\sum_{i\in I}k_{i,n}^2L_iX_ndt
,\qquad n\geq1
\end{equation}

Let $H:=l^2(\mathbb R^d)$ denote the state space and $\|\cdot\|$ its
norm.

\begin{definition}
Given an initial condition $x\in H$, a weak solution of non-linear
system~\eqref{e:itodiffmodel} (respectively of linear
system~\eqref{e:itodifflinearmodel}) in $H$ is a filtered probability
space $(\Omega, \mathcal F_\infty, (\mathcal F_t)_{t\geq0}, P)$, along
with a family of Brownian motions $W$, and a stochastic process $X$
such that
\begin{enumerate}[\hspace{2em}\it i.]
\item $W=(W_{i,n})_{i\in I, n\in\mathbb Z}$, is a family of
  $d$-dimensional Brownian motions on $(\Omega, \mathcal F_\infty,
  (\mathcal F_t)_{t\geq0}, P)$ adapted to the filtration and
  symmetric with respect to $\tau$;
\item $X=(X_n)_{n\geq1}$ is an $H$-valued stochastic process on
  $(\Omega, \mathcal F_\infty, (\mathcal F_t)_{t\geq0}, P)$ with
  continuous adapted components;
\item the following integral form of non-linear
  equation~\eqref{e:itodiffmodel} holds for all $n\geq1$ and all
  $t\geq0$,
\begin{multline}\label{e:itointmodel}
X_n(t)=x_n+\sum_{i\in I}\Bigl\{\int_0^tk_{i,n}B_i(X_{n+r_i}(s),X_{n+h_i}(s))ds\\
+\int_0^t\sigma k_{i,n}B_i\bigl(X_{n+r_i}(s),dW_{i,n+h_i}(s)\bigr)-\int_0^t\frac{\sigma^2}2k_{i,n}^2L_iX_n(s)ds\Bigr\}
\end{multline}
(respectively, the following integral form of linear
equation~\eqref{e:itodifflinearmodel} holds for all $n\geq1$ and all
$t\geq0$)
\begin{equation}\label{e:itointlinearmodel}
X_n(t)=x_n+\sum_{i\in I}\Bigl\{\int_0^t\sigma k_{i,n}B_i\bigl(X_{n+r_i}(s),dW_{i,n+h_i}(s)\bigr)-\int_0^t\frac{\sigma^2}2k_{i,n}^2L_iX_n(s)ds\Bigr\}
\end{equation}
\end{enumerate}
\end{definition}

The next proposition shows that what is defined above is actually a
solution of the Stratonovich formulation of the non-linear system.

\begin{proposition}\label{prop:equivalencestratito}
If $X$ is a weak solution of the non-linear
system~\eqref{e:itodiffmodel}, the Stratonovich integrals
\begin{equation}\label{e:stratonovich_integral_term}
\int_0^t\sigma k_{i,n}B_i\bigl(X_{n+r_i}(s),\strt dW_{i,n+h_i}(s)\bigr)
\end{equation}
are well defined and equal to
\begin{equation}\label{e:strintegr1}
\int_0^t\sigma k_{i,n}B_i\bigl(X_{n+r_i}(s), dW_{i,n+h_i}(s)\bigr)
-\int_0^t\frac{\sigma^2}2k_{i,n}^2L_iX_n(s)ds
\end{equation}
Hence $X$ satisfies the Stratonovich equations \eqref{e:mainmodel}.
\end{proposition}

\begin{proof}
We write components explicitly, in particular
$B_i=\bigl(B_i^{\alpha,\beta,\gamma}\bigr)_{\alpha,\beta,\gamma}$ and
$L_i^{\alpha,\beta}=\sum_{\gamma,\delta}B_i^{\alpha,\gamma,\delta}B_i^{\beta,\gamma,\delta}$.
Component $\alpha$ of~\eqref{e:stratonovich_integral_term} rewrites
\begin{equation}\label{e:strat_int_term_component}
\sigma k_{i,n}\sum_{\beta,\gamma}B_i^{\alpha,\beta,\gamma}\int_0^tX_{n+r_i}^\beta(s)\strt dW_{i,n+h_i}^\gamma(s)
\end{equation}
The stochastic integral can be rewritten in It\=o form
\begin{equation}
\int_0^tX_{n+r_i}^\beta(s)\strt dW_{i,n+h_i}^\gamma(s)
=\int_0^tX_{n+r_i}^\beta(s)dW_{i,n+h_i}^\gamma(s)+\frac12\left[X_{n+r_i}^\beta,W_{i,n+h_i}^\gamma\right]_t
\end{equation}
we only need to compute the quadratic covariation term. When
$n+r_i\leq0$, this is zero, while if $n+r_i>0$, by
writing~\eqref{e:itointmodel} with $n+r_i$ in place of $n$, we find
\begin{multline}
\left[X_{n+r_i}^\beta,W_{i,n+h_i}^\gamma\right]_t\\
=\left[\sum_{j\in I}\sigma k_{j,n+r_i}\sum_{\delta,\eta}B_j^{\beta,\delta,\eta}\int_0^tX_{n+r_i+r_j}^\delta dW_{j,n+r_i+h_j}^\eta,W_{i,n+h_i}^\gamma\right]_t\\
=\sigma\sum_{j\in I}k_{j,n+r_i}\sum_{\delta,\eta}B_j^{\beta,\delta,\eta}\int_0^tX_{n+r_i+r_j}^\delta(s) \,\,d\!\left[W_{j,n+r_i+h_j}^\eta,W_{i,n+h_i}^\gamma\right]_s
\end{multline}
The very last quadratic covariation differential can be $ds$ or 0,
depending on whether the two particular BM involved are equal or
independent. They are clearly independent when
$j\not\in\{i,\tilde\imath\}$. They are independent also when $j=i$,
since $r_i\neq0$. Finally, they are a.s.~equal when
$j=\tilde\imath$ and $\eta=\gamma$ by conditions~\eqref{e:alias_h}
and~\eqref{e:cancellations}. We get
\begin{multline}
\left[X_{n+r_i}^\beta,W_{i,n+h_i}^\gamma\right]_t
=\sigma k_{\tilde\imath,n+r_i}\sum_\delta B_{\tilde\imath}^{\beta,\delta,\gamma}\int_0^tX_{n+r_i+r_{\tilde\imath}}^\delta(s)ds\\
=-\sigma k_{i,n}\sum_{\delta}B_i^{\delta,\beta,\gamma}\int_0^tX_n^\delta(s)ds
\end{multline}
where we used conditions~\eqref{e:cancellations_k},
\eqref{e:alias_B}, and~\eqref{e:cancellations_r}. Putting all
together
\begin{multline}
\left[\int_0^t\sigma k_{i,n}B_i\bigl(X_{n+r_i}(s),\strt dW_{i,n+h_i}(s)\bigr)
-\int_0^t\sigma k_{i,n}B_i\bigl(X_{n+r_i}(s), dW_{i,n+h_i}(s)\bigr)\right]^\alpha\\
=\frac\sigma2 k_{i,n}\sum_{\beta,\gamma}B_i^{\alpha,\beta,\gamma}\left[X_{n+r_i}^\beta,W_{i,n+h_i}^\gamma\right]_t
=-\frac{\sigma^2}2 k_{i,n}^2\sum_{\beta,\gamma}B_i^{\alpha,\beta,\gamma}\sum_{\delta}B_i^{\delta,\beta,\gamma}\int_0^tX_n^\delta(s)ds\\
=-\int_0^t\frac{\sigma^2}2 k_{i,n}^2\sum_{\delta}X_n^\delta(s)ds\sum_{\beta,\gamma}B_i^{\alpha,\beta,\gamma}B_i^{\delta,\beta,\gamma}
=-\left[\int_0^t\frac{\sigma^2}2 k_{i,n}^2L_iX_n(s)ds\right]^\alpha
\end{multline}
The latter is correct also when $n+r_i\leq0$, since in that case
$k_{i,n}=0$.
\end{proof}

\begin{definition}
Given an initial condition $x\in H$, an energy controlled solution of
the non-linear system~\eqref{e:itodiffmodel} or the linear
system~\eqref{e:itodifflinearmodel} is a weak solution of the same
system of equations in the class $L^\infty(\Omega\times[0,\infty);H)$. In
particular, if $\|X\|_{L^\infty}=\|x\|$, it is called a Leray solution.
\end{definition}

\section{Girsanov transformation}\label{sec:3-girsanov}

We turn our attention to the terms $X_{n+h_i}ds+\sigma dW_{i,n+h_i}$
in equation~\eqref{e:itodiffmodel}. We would like that, under a new
probability measure $Q$, these were the differentials $\sigma
dY_{i,n+h_i}$, where $Y$ is again a family of $d$-dimensional Brownian
motions symmetric with respect to $\tau$. To do so, we state an
infinite-dimensional version of Girsanov theorem whose proof can be
found in Da Prato \emph{et al.}~\cite{DaPFlaPriRoe}.

\begin{theorem}\label{thm:girsanov}
On a filtered space $(\Omega, \mathcal F_\infty, (\mathcal
F_t)_{t\geq0}, P)$, let $(W_j)_{j\in\mathbb N}$ be a sequence of
1-dimensional adapted independent Brownian motions and let
$X=(X_j)_{j\in\mathbb N}$ be a sequence of adapted semimartingales
such that $\mathbb E\sum_jX_j^2(t)<\infty$ for all $t\geq 0$.

Let $Y_j(t):=\int_0^tX_j(s)ds+W_j(t)$ for $j\in\mathbb N$.  Put, for
$0\leq t\leq\infty$
\[
M_t=\exp\biggl\{-\int_0^t\sum_jX_j(s)dW_j(s)-\frac12\int_0^t\sum_jX_j^2(s)ds\biggr\}
\]
Fix $0<T\leq\infty$. Suppose $\mathbb E[M_T]=1$, then $M$ is a closed
martingale on $[0,T]$ and the density $\frac{dQ}{dP}=M_T$ defines a
  new probability measure $Q$ on $\mathcal F_T$ under which
  $(Y_j)_{j\in\mathbb N}$ is a sequence of independent Brownian
  motions on $[0,T)$.

Moreover, to prove that $\mathbb E[M_T]=1$, Novikov condition can be used,
namely it is enough to prove that
\begin{equation}
\mathbb E\Biggl[\exp\biggl\{\frac12\int_0^T\sum_jX_j^2(s)ds\biggr\}\Biggr]<\infty
\end{equation}
\end{theorem}

By virtue of Theorem~\ref{thm:girsanov} it is quite easy to verify
that, by changing the probability measure and the family of Brownian
motions, any Leray solution of the non-linear
system~\eqref{e:itodiffmodel} can be transformed into a Leray solution
of the auxiliary linear system~\eqref{e:itodifflinearmodel}. The
reverse is also true.

In our notation, $X$ will always denote the solution process; $P$ and
$Q$ will denote the probability measures for the non-linear and linear
systems respectively; $P_T$ and $Q_T$ their restrictions to $\mathcal
F_T$; $W$ and $Y$ will denote the two associated families of Brownian
motions.

The relation between $W$ and $Y$ is ensured by the following
definition. For $i\in I$, $n\in\mathbb N$ and $t\geq0$, let
\begin{equation}\label{e:intY}
Y_{i,n}(t)=\int_0^t\frac1\sigma X_n(s)ds+W_{i,n}(t)
\end{equation}
Suppose we can define the two martingales
\begin{align}\label{e:defZtandtilde}
Z_t&=-\int_0^t\sum_{\substack{i\in I^*\\n\in\mathbb N}}\bigl\langle\sigma^{-1}X_{n+h_i}(s),dW_{i,n+h_i}(s)\bigr\rangle, \\
\tilde Z_t&=\int_0^t\sum_{\substack{i\in I^*\\n\in\mathbb N}}\bigl\langle\sigma^{-1}X_{n+h_i}(s),dY_{i,n+h_i}(s)\bigr\rangle
\end{align}
Then it easy to verify that
\[
\tilde Z_t
=-Z_t+\frac1{\sigma^2}\int_0^t\sum_{\substack{i\in I^*\\n\in\mathbb N}}|X_{n+h_i}(s)|^2ds
=-Z_t+[Z,Z]_t
=-Z_t+[\tilde Z,\tilde Z]_t
\]
so that $Z_t-\frac12[Z,Z]_t=-\tilde Z_t+\frac12[\tilde Z,\tilde Z]_t$.
We will then pose
\begin{align}\label{e:QPdensity}
\frac{dQ_t}{dP_t}
&=\exp\{Z_t-\frac12[Z,Z]_t\}, &
\frac{dP_t}{dQ_t}
&=\exp\{\tilde Z_t-\frac12[\tilde Z,\tilde Z]_t\} &
\end{align}
We are now ready to make a precise statement
\begin{proposition}\label{prop:nonlinearsol_solveslinear}
Suppose $(\Omega, \mathcal F_\infty, (\mathcal F_t)_{t\geq0}, P, W,
X)$ is a Leray solution of the non-linear
system~\eqref{e:itodiffmodel}. Fix any $0<T<\infty$.  Let $Q_T$ be the
measure on $\mathcal F_T$ defined by~\eqref{e:QPdensity} and let $Y$
be the family of Brownian motions defined by~\eqref{e:intY}.

Then $Q_T$ is a probability measure and $(\Omega, \mathcal F_T,
(\mathcal F_t)_{0\leq t\leq T}, Q_T, Y, X)$ is a Leray solution of the
linear system~\eqref{e:itodifflinearmodel} on $[0,T]$.
\end{proposition}
\begin{proof}
By~\eqref{e:intY}, equation~\eqref{e:itointmodel} is equivalent to
equation~\eqref{e:itointlinearmodel} with $W$ replaced by $Y$. We need
to prove that $Q_T$ is a probability measure and that $Y$ is a family
of Brownian motions symmetric w.r.t.~$\tau$, hence we apply
Theorem~\ref{thm:girsanov}.

The sequences we use are $\frac1\sigma X_{n+h_i}^j$ and
$W_{i,n+h_i}^j$, both for $1\leq j\leq d$, $n\in\mathbb N$ and $i\in
I^*$. (Notice that we use $I^*$ instead of $I$ since we need the
independence of the Brownian motions.)

By Leray property and the finiteness of $I^*$, we get a very strong
bound
\begin{equation}
\sum_{i\in I^*}\sum_{n\geq1}|X_{n+h_i}|^2(t)
\leq|I^*|\|x\|^2\qquad\text{a.s.~for all }t>0
\end{equation}
by which we immediately deduce both that $Z_t$
in~\eqref{e:defZtandtilde} is well-defined and, by the finiteness of
$T$, that
\begin{equation}\label{e:quadr_var_Z_bound_T}
[Z,Z]_T
=[\tilde Z,\tilde Z]_T
=\int_0^T\sum_{i\in I^*}\sum_{n\geq1}\frac1{\sigma^2}|X_{n+h_i}|^2(s)ds
\leq\frac{|I^*|\|x\|^2T}{\sigma^2}
,\qquad{a.s.}
\end{equation}
hence Novikov condition holds, namely
\begin{equation}\label{e:novikov_T}
\mathbb E\bigl[e^{\frac12[Z,Z]_T}\bigr]
=\mathbb E\bigl[e^{\frac12[\tilde Z,\tilde Z]_T}\bigr]
\leq \exp\frac{|I^*|\|x\|^2T}{2\sigma^2}
<\infty
\end{equation}
Finally, $Y$ is symmetric w.r.t.~$\tau$ since both $P$-a.s.\ and
$Q_T$-a.s.
\begin{equation}
Y_{\tilde\imath,n}(t)
=\int_0^t\frac1\sigma X_n(s)ds+W_{\tilde\imath,n}(t)
=\int_0^t\frac1\sigma X_n(s)ds+W_{i,n}(t)
=Y_{\tilde\imath,n}(t)\qedhere
\end{equation}
\end{proof}
The converse is also true. We give it without proof since it is almost
identical to the one above.
\begin{proposition}\label{prop:linearsol_solvesnonlinear}
Suppose $(\Omega, \mathcal F_\infty, (\mathcal F_t)_{t\geq0}, Q, Y,
X)$ is a Leray solution of the linear
system~\eqref{e:itodifflinearmodel}. Fix any $0<T<\infty$.  Let $P_T$
be the measure on $\mathcal F_T$ defined by the second one of~\eqref{e:QPdensity} and let $W$ be the family of Brownian motions
defined by~\eqref{e:intY}.

Then $P_T$ is a probability measure and $(\Omega, \mathcal F_T,
(\mathcal F_t)_{0\leq t\leq T}, P_T, W, X)$ is a Leray solution of the
non-linear system~\eqref{e:itodiffmodel} on $[0,T]$.
\end{proposition}

\begin{remark}\label{rem:P_Q_equivalence}
By Carath\'eodory theorem, the family of probability measures
$(Q_T)_{T\geq0}$ extends in a unique way to some probability measure
$Q$ on $\mathcal F_\infty$ and the same stands for $(P_T)_{T\geq0}$.
Hence solutions of non-linear and linear systems are associated also on
infinite time span. From now on we will drop the $_T$ and use the
symbols $P$ and $Q$ with this meaning. One should anyway keep in mind
that while $P_T$ and $Q_T$ (and hence $P$ and $Q$) are equivalent on
$\mathcal F_T$ for any finite $T$, they are not in general equivalent
on $\mathcal F_\infty$.
\end{remark}

\section{Closed equation for \texorpdfstring{$\mathbb E^Q\left[|X_n(t)|^2\right]$}{second moments} and uniqueness}\label{sec:4-uniqueness}

Denote by $\mathbb E^Q$ the mathematical expectation on $(\Omega,
\mathcal F, Q)$. It turns out that if $L_i$ is the identity for all
$i\in I$, then $\mathbb E^Q\left[|X_n(t)|^2\right]$ satisfies a closed
linear differential equation which will shed new light on the
behaviour of solutions, in particular by giving an easy way to prove
uniqueness.
\begin{proposition}
Suppose $(\Omega, \mathcal F_\infty, (\mathcal F_t)_{t\geq0}, Q, Y,
X)$ is an energy controlled solution of the linear
system~\eqref{e:itodifflinearmodel}.

Then for all $n\geq1$ and $t\geq0$
\begin{equation}\label{e:diff2ndmomentli}
\frac d{dt}\mathbb E^Q\left[|X_n|^2\right]
=-\sum_{i\in I}\sigma^2k_{i,n}^2\mathbb E^Q\bigl\langle X_n,L_iX_n\bigr\rangle 
+\sum_{i\in I}\sigma^2k_{i,n}^2\mathbb E^Q\langle X_{n+r_i},L_{\tilde\imath}X_{n+r_i}\rangle
\end{equation}
\end{proposition}
\begin{proof}
We start by computing the quadratic variation of $X_n$. We
use~\eqref{e:itointlinearmodel} and the independence of $Y_{i,n+h_i}$
and $Y_{j,n+h_j}$ when $j\neq i$. (If $j=\tilde\imath$, then
$n+h_j=n+h_i-r_i\neq n+h_i$.)
\begin{multline*}
\frac{d[X_n,X_n]_t}{\sigma^2}
=\sum_{i\in I}\sum_{j\in I}k_{i,n}k_{j,n}d\biggl[\int_0^tB_i(X_{n+r_i},dY_{i,n+h_i}),\int_0^tB_j(X_{n+r_j},dY_{j,n+h_j})\biggr]_t\\
=\sum_{i\in I}k_{i,n}^2d\biggl[\int_0^tB_i(X_{n+r_i},dY_{i,n+h_i}),\int_0^tB_i(X_{n+r_i},dY_{i,n+h_i})\biggr]_t\\
=\sum_{i\in I}k_{i,n}^2\sum_{\alpha,\beta,\gamma,\delta,\epsilon}B_i^{\alpha,\beta,\gamma}B_i^{\alpha,\delta,\epsilon}X_{n+r_i}^\beta X_{n+r_i}^\delta d\bigl[Y_{i,n+h_i}^\gamma,Y_{i,n+h_i}^\epsilon\bigr]_t\\
=\sum_{i\in I}k_{i,n}^2\sum_{\beta,\delta}\sum_{\alpha,\gamma}B_{\tilde\imath}^{\beta,\alpha,\gamma}B_{\tilde\imath}^{\delta,\alpha,\gamma}X_{n+r_i}^\beta X_{n+r_i}^\delta dt\\
=\sum_{i\in I}k_{i,n}^2\langle X_{n+r_i},L_{\tilde\imath}X_{n+r_i}\rangle dt
\end{multline*}
We also used~\eqref{e:alias_B} and the definition of
$L_i^{\alpha,\beta}=\sum_{\gamma,\delta}B_i^{\alpha,\gamma,\delta}B_i^{\beta,\gamma,\delta}$.

Now we are able to compute the differential of $|X_n(t)|^2$
\begin{multline}\label{e:differentialXnsquare}
d(|X_n(t)|^2)
=2\langle X_n(t),dX_n(t)\rangle+d[X_n,X_n]_t\\
=2\sum_{i\in I}\sigma k_{i,n}\bigl\langle X_n,B_i(X_{n+r_i},dY_{i,n+h_i})\bigr\rangle
-\sum_{i\in I}\sigma^2k_{i,n}^2\bigl\langle X_n,L_iX_n\bigr\rangle dt+\\
+\sum_{i\in I}\sigma^2k_{i,n}^2\langle X_{n+r_i},L_{\tilde\imath}X_{n+r_i}\rangle dt
\end{multline}
By the definition of energy controlled solution, $|X_n(t)|\leq
\|X(t)\|\leq C$ almost surely, so in particular $\mathbb
E^Q\int_0^t|X_n|^2(s)|X_{n+r_i}|^2(s)ds<\infty$ and hence the first
term above is a martingale, with mathematical expectation zero. If we
now take the mathematical expectation of the integral form
of~\eqref{e:differentialXnsquare} we get the integral form
of~\eqref{e:diff2ndmomentli} and we are finished.
\end{proof}
When all the $L_i$ are the identity,
equation~\eqref{e:diff2ndmomentli} becomes a linear closed
differential equation
\begin{equation}\label{e:diff2ndmomentclosed}
\frac d{dt}\mathbb E^Q\left[|X_n|^2\right]
=-\sum_{i\in I}\sigma^2k_{i,n}^2 \mathbb E^Q\left[|X_n|^2\right]
+\sum_{i\in I}\sigma^2k_{i,n}^2 \mathbb E^Q\left[|X_{n+r_i}|^2\right]
\end{equation}
We notice that this system of equations is of a peculiar kind, with
negative diagonal and non-negative off-diagonal entries, thus
suggesting a connection to the Kolmogorov equations for
continuous-time Markov chains on the positive integers. We investigate
this relation presently.

Denote by $\Pi=(\pi_{m,n})_{m,n\geq1}$ the infinite matrix associated
to this system: for $n,m\geq1$ and $m\neq n$ let
\begin{equation}\label{e:def_Pi}
\begin{cases}
\pi_{n,m}
:=\displaystyle\sum_{\substack{i\in I\\r_i=m-n}}\sigma^2k_{i,n}^2
=\sigma^2\lambda^{2n}\displaystyle\sum_{\substack{i\in I_n\\r_i=m-n}}k_i^2\\[5ex]
\pi_{n,n}
:=-\displaystyle\sum_{i\in I}\sigma^2k_{i,n}^2
=-\sigma^2\lambda^{2n}\displaystyle\sum_{i\in I_n}k_i^2
=:-\pi_n
\end{cases}
\end{equation}
\begin{remark}\label{rem:I_n_stable}
Recall that for $n\geq1$, $I_n:=\{i\in I:n+r_i\geq1, n+h_i\geq1\}$, so
that 
\[
I_1\subset I_2\subset\dots\subset I_{n_0}=I_{n_0+1}=\dots=I.
\]
where $n_0\geq1-\min\{r_i,h_i:i\in I\}$. Hence for example
$\pi_n=O(\lambda^{2n})$ as $n\to\infty$.
\end{remark}

\begin{corollary}\label{cor:2ndQmomentsolvesforward}
Suppose that $(Q,X)$ is an energy controlled solution of
equation~\eqref{e:itodifflinearmodel} and that all the $L_i$ are the
identity. Then $u=(u_n)_{n\geq1}$ defined by $u_n(t)=\mathbb
E^Q\left[|X_n(t)|^2\right]$ is a non-negative solution of the Cauchy
problem
\begin{equation}\label{e:2ndmom_forward}
\begin{cases}
u'=u\Pi \\
u_n(0)=|x_n|^2 & n\geq1\\
\end{cases}
\end{equation}
in the class $L^\infty\bigl([0,\infty);l^1\bigr)$.
\end{corollary}
\begin{proposition}\label{pro:PiisQmatrix}
The infinite matrix $\Pi$ defined above is the stable, conservative
q-matrix of a continuous-time Markov chain on the positive
integers. Moreover $\Pi$ is symmetric.
\end{proposition}
\begin{proof}
$\Pi$ is a stable q-matrix if $\pi_{n,n}<0$ for all $n$,
  $\pi_{n,m}\geq0$ whenever $n\neq m$ and for all $n$
\begin{equation}
\sum_{m:m\neq n}\pi_{n,m}\leq\pi_n,
\end{equation}
moreover it is conservative if the latter holds with equality.

The first two conditions are obvious, and the third one follows from
the fact that the sets $\{i\in I:n+r_i=m\}$ with $m\geq1$, $m\neq n$
form a partition of $\{i\in I:n+r_i\geq1, k_{i,n}\neq0\}$.

Finally, for $m\neq n$,
\begin{equation}
 \pi_{n,m}
=\sum_{\substack{\tilde\imath\in I\\r_{\tilde\imath}=m-n}}\sigma^2k_{\tilde\imath,n}^2
=\sum_{\substack{i\in I\\r_i=n-m}}\sigma^2k_{i,n-r_i}^2
=\sum_{\substack{i\in I\\r_i=n-m}}\sigma^2k_{i,m}^2
=\pi_{m,n}\qedhere
\end{equation}
\end{proof}

The q-matrix of a continuous-time Markov chain is associated to the
forward and backward Kolmogorov equations, namely
\begin{align}\label{e:ctmc_forward}
u'&=u\Pi\\
u'&=\Pi u
\end{align}
The transition probabilities of the Markov chain $p_{n,m}(t)$ solve
both equations, in the classes $l^1$ and $l^\infty$ respectively, with
fixed $n$ and $m$ respectively and with initial condition
$u_{n,m}(0)=\delta_{n,m}$.

These equations always have at least one shared ``special'' solution
$f_{i,j}(t)$, which is a transition function, and is called the
minimal solution. They do not always have uniqueness of
solutions. Here it will be important that there is uniqueness for the
forward equation and not for the backward. The key information is that
the q-matrix is symmetric.

\begin{lemma}\label{lem:uniqforwardsymm}
Suppose $\Pi$ is a stable and \emph{symmetric} q-matrix. Consider the
forward equations with zero initial condition. Then the only
non-negative solution in $L^\infty([0,\infty);l^1)$ is zero.

More in general, given any non-negative $l^1$ initial condition, there
a unique solution in the same class.
\end{lemma}

\begin{proof}
For the first part, we follow the classical approach by Laplace
transform, introduced by Feller~\cite{Feller57}. Let $\rho$ be such a
solution. For all $n\geq1$ and $t\geq0$ we have
\[
\begin{cases}
\rho'_n(t)=\sum_k\rho_k(t)\pi_{k,n}\\
\rho_n(t)\geq0\\
\rho_n(0)=0\\
\sum_k\rho_k(t)\leq C
\end{cases}
\]
For all $n\geq1$, let $z_n=\int_0^\infty e^{-t}\rho_n(t)dt$. Clearly
$\sum_nz_n\leq C$, so we can choose $m$ such that $z_m\geq z_k$ for
all $k$.

Notice that, since $\Pi$ is stable and symmetric
\[
|\rho'_m(t)|
=|-\pi_m\rho_m(t)+\sum_{k\neq m}\pi_{m,k}\rho_k(t)|
\leq\pi_m\rho_m(t)+\pi_mC
\leq2C\pi_m<\infty
\]
hence we can integrate by parts and use symmetry and stability again to get
\begin{multline*}
z_m
=\int_0^\infty e^{-t}\rho'_m(t)dt
=\int_0^\infty e^{-t}\sum_k\rho_k(t)\pi_{k,m}dt
=\sum_kz_k\pi_{k,m}\\
=-z_m\pi_m+\sum_{k\neq m}\pi_{m,k}z_k
\leq-z_m\pi_m+z_m\sum_{k\neq m}\pi_{m,k}
\leq0
\end{multline*}
We conclude that $z_n=0$ and $\rho_n\equiv0$ for all $n$.

\medskip
For the general case, let $f_{i,j}(t)$ be the minimal solution of
$\Pi$ and let $u^0$ be a non-negative, $l^1$ initial condition. Then
$u_n(t)=\sum_{i\geq1}u^0_if_{i,n}(t)$ is a solution in the required
class. Let $v$ be another such solution and let $\rho=v-u$. By a
forward integral recursion (FIR) approach it is easy to show that the
minimality of $f$ passes to $u$, in that $v_n(t)\geq u_n(t)$. (See for
example Anderson~\cite{Anderson}, Theorem 2.2.2.) So $\rho$ is a
solution of the same problem, but with null initial condition and the
first part of the lemma applies.
\end{proof}
We have finally collected all elements to prove uniqueness of
solutions.
\begin{theorem}\label{thm:uniqueness_linear}
Suppose $L_i$ is the identity matrix for all $i\in I$.
Then there is strong uniqueness for the linear
system~\eqref{e:itodifflinearmodel} in the class of
$L^\infty(\Omega\times[0,\infty);H)$ solutions.
\end{theorem}
\begin{proof}
By linearity of~\eqref{e:itodifflinearmodel} it is enough to prove
that when the initial condition is $x=0$ there is no non-trivial
solution. Suppose $(Q,X)$ is any energy controlled solution with zero
initial condition, then by
Corollary~\ref{cor:2ndQmomentsolvesforward},
Proposition~\ref{pro:PiisQmatrix} and the first part of
Lemma~\ref{lem:uniqforwardsymm}, $\mathbb
E^Q\left[|X_n(t)|^2\right]=0$ for all $n$ and $t$, hence $X=0$ a.s.
\end{proof}

\begin{remark}
This result applies seamlessly also to the case of $L^\infty$,
non-anticipative, random initial conditions.
\end{remark}
Uniqueness of solutions for the auxiliary linear system is then
inherited by the original non-linear system, but in a weakened form.
\begin{theorem}\label{thm:uniqueness_nonlinear}
Suppose $L_i$ is the identity matrix for all $i\in I$ and let $T>0$. Then
there is uniqueness in law for the non-linear
system~\eqref{e:itodiffmodel} in the class of Leray
$L^\infty(\Omega\times[0,T];H)$ solutions.
\end{theorem}

\begin{proof}
Suppose we are given two solutions $(P^{(1)},W^{(1)},X^{(1)})$ and
$(P^{(2)},W^{(2)},X^{(2)})$. We want to prove that
\begin{multline}
\mathbb E^{P^{(1)}}[f(X^{(1)}(t_1),X^{(1)}(t_2),\dots,X^{(1)}(t_n))]\\
=\mathbb E^{P^{(2)}}[f(X^{(2)}(t_1),X^{(2)}(t_2),\dots,X^{(2)}(t_n))]
\end{multline}
where $f$ is any bounded measurable real function on $H^n$ and
$t_1,t_2,\dots,t_n\in[0,T]$. By
Proposition~\ref{prop:nonlinearsol_solveslinear} and the first one
of~\eqref{e:QPdensity} we have that, for $j=1,2$
\begin{multline}\label{e:expecvalfofx}
\mathbb E^{P^{(j)}}[f(X^{(j)}(t_1),X^{(j)}(t_2),\dots,X^{(j)}(t_n))]\\
=\mathbb E^{Q^{(j)}}[\exp\{-Z^{(j)}_T+\frac12[Z^{(j)},Z^{(j)}]_T\}f(X^{(j)}(t_1),X^{(j)}(t_2),\dots,X^{(j)}(t_n))]
\end{multline}
Where $Z^{(j)}$ is defined by~\eqref{e:defZtandtilde}.  By
Theorem~\ref{thm:uniqueness_linear}
equation~\eqref{e:itodifflinearmodel} has strong uniqueness, hence we
can apply an infinite-dimensional version of Yamada-Watanabe theorem
(see Revuz and Yor~\cite{RevuzYor} or Pr\'ev\^ot and
R\"ockner~\cite{PreRoe}) to deduce that the laws of
$(X^{(1)},W^{(1)})$ and $(X^{(2)},W^{(2)})$ on $C([0,T];\mathbb
R^{2d})^{\mathbb N}$ are equal, under $Q^{(1)}$ and $Q^{(2)}$
respectively.

Then of course we can include also $Z^{(i)}$ by their definition and
conclude that $(X^{(1)},W^{(1)},Z^{(1)})$ under $Q^{(1)}$ and
$(X^{(2)},W^{(2)},Z^{(2)})$ under $Q^{(2)}$ have the same law,
yielding in particular that~\eqref{e:expecvalfofx} does not depend on
$j$.
\end{proof}

\section{Existence of solutions}\label{sec:5-existence}

In this section we prove strong existence of solutions for the linear
auxiliary model and deduce weak existence for the non-linear
model. The approach is by finite-dimensional approximation and follows
Pardoux~\cite{Pardoux75phd} and Krylov and
Rozovski\u\i~\cite{KryRoz79}. The border term of the
finite-dimensional systems is chosen so that energy conservation
holds, giving a strong tool to prove convergence.

\begin{theorem}\label{thm:strong_existence_linear}
Let $(\Omega, \mathcal F_\infty, (\mathcal F_t)_{t\geq0}, Q)$ be a
filtered probability space. Let $Y$ be a family of adapted
$d$-dimensional Brownian motions symmetric with respect to $\tau$.
Given an initial condition $\chi\in L^\infty((\Omega,\mathcal F_0);H)$ and
$T>0$, there exists at least an $H$-valued stochastic process $X$ with
continuous adapted components, such that $Q$-almost surely
$\|X(t)\|\leq\|\chi\|$ for all $t\geq0$ and for all $n\geq1$ and all
$t\geq0$,
\begin{equation}\label{e:itointlinearmodel_random_ic}
X_n(t)=\chi_n+\sum_{i\in I}\Bigl\{\int_0^t\sigma k_{i,n}B_i\bigl(X_{n+r_i}(s),dW_{i,n+h_i}(s)\bigr)-\int_0^t\frac{\sigma^2}2k_{i,n}^2L_iX_n(s)ds\Bigr\}
\end{equation}
Such a process is called strong Leray solution.
\end{theorem}
\begin{proof}
For every positive $N$, let $A_N=\{1,\dots,N\}$ and consider the
finite dimensional stochastic linear system
\begin{equation}
\begin{cases}
dX^{(N)}_n=\sum_{i\in I}k_{i,n}B_i(X^{(N)}_{n+r_i},\sigma dY_{i,n+h_i})-\frac{\sigma^2}2\sum_{i\in I}k_{i,n}^2L_iX^{(N)}_ndt
, & n\in A_N\\
X^{(N)}_n(0)=\chi_n
, & n\in A_N\\
X^{(N)}_n\equiv0
, & n\in \mathbb Z\setminus A_N
\end{cases}
\end{equation}
This system has a unique global strong solution $X^{(N)}$. By the
local conservativity, we can prove that the $l^2$ norm is
$Q$-a.s.~constant. In particular we notice that
equation~\eqref{e:differentialXnsquare} applies to $X^{(N)}_n$, for
$n\in A_N$ without modifications. Then we sum on $n$ and apply
Lemma~\ref{le:some_alg_conseq} to get
\begin{multline*}
\sum_{n\in A_N}d(|X^{(N)}_n(t)|^2)
=2\sum_{I\times\mathbb Z}\sigma k_{i,n}\bigl\langle X^{(N)}_n,B_i(X^{(N)}_{n+r_i},dY_{i,n+h_i})\bigr\rangle\\
-\sum_{I\times\mathbb Z}\sigma^2k_{i,n}^2\bigl\langle X^{(N)}_n,L_iX^{(N)}_n\bigr\rangle dt
+\sum_{I\times\mathbb Z}\sigma^2k_{i,n}^2\langle X^{(N)}_{n+r_i},L_{\tilde\imath}X^{(N)}_{n+r_i}\rangle dt\\
=2\sigma\sum_\Delta\biggl\{k_{i,n}\bigl\langle X^{(N)}_n,B_i(X^{(N)}_{n+r_i},dY_{i,n+h_i})\bigr\rangle
+k_{\tilde\imath,\tilde n}\bigl\langle X^{(N)}_{\tilde n},B_{\tilde\imath}(X^{(N)}_{\tilde n+r_{\tilde\imath}},dY_{{\tilde\imath},\tilde n+h_{\tilde\imath}})\bigr\rangle\biggr\}\\
-\sum_{I\times\mathbb Z}\sigma^2k_{i,n}^2\bigl\langle X^{(N)}_n,L_iX^{(N)}_n\bigr\rangle dt
+\sum_{I\times\mathbb Z}\sigma^2k_{{\tilde\imath},\tilde n}^2\langle X^{(N)}_{\tilde n},L_{\tilde\imath}X^{(N)}_{\tilde n}\rangle dt=0
\end{multline*}
Thus
\begin{equation}\label{e:galerkin_energy_isometry}
\sum_{n\in A_N}|X^{(N)}_n(t)|^2
=\sum_{n\in A_N}|\chi_n|^2
\leq\|\chi\|_{L^\infty(\Omega;H)}^2
,\qquad \forall t\geq0,\quad Q\text{-a.s.}
\end{equation}
meaning in particular that the sequence $X^{(N)}$ is bounded in
$L^\infty(\Omega\times[0,T];H)$. Hence there exists $X$ in the same
space and a sequence $N_k\uparrow\infty$ such that
$X^{(N_k)}\xrightarrow{w*} X$ as $k\to\infty$. A fortiori there is
also weak convergence in $L^2(\Omega\times[0,T];H)$.

Let $\mathbb X$ denote the subspace of $L^2(\Omega\times[0,T];H)$ of
the progressively measurable processes, then $X^{(N)}\in\mathbb X$ for
all $N$. The space $\mathbb X$ is complete, hence it is a closed
subspace of $L^2$ in the strong topology, hence it is also closed in
the weak topology, so $X$ must be progressively measurable.

Now we want to show that $X$ indeed satisfies
equation~\eqref{e:itointlinearmodel} with $Y$ in place of $W$, as each
of the $X^{(N)}$'s does. Fix $t\in[0,T]$, $i\in I$, $n\geq1$ and
$l,m\in\{1,2,\dots,d\}$. The map
\[
V\to\int_0^tV^l_{n+r_i}(s)dY^m_{i,n+h_i}(s)
\]
is a linear strongly continuous map from $\mathbb X$ to
$L^2(\Omega;\mathbb R)$, so it is weakly continuous.
Equation~\eqref{e:itointlinearmodel}, written component-wise, reduces
to a finite sum of one-dimensional stochastic integrals like the one
above, hence we can pass to the limit and so $X$ solves the same
equations. A posteriori, from these integral equations, it follows that
there is a modification such that all components are continuous.

Finally we prove the Leray property.  Consider the product
measure $\mu=Q\times\mathcal L$ on $\Omega\times[0,T]$. Let
$\epsilon>0$, let
$A=\{(\omega,t):\|X(\omega,t)\|\geq\|\chi(\omega)\|+\epsilon\}$ and
let $U=\frac X{\|X\|}\mathbbm 1_A\in L^2(\Omega\times[0,T];H)$. Then
\[
\langle X,U\rangle_{L^2}
=\int \|X\|\mathbbm 1_Ad\mu
\geq \int_A\|\chi\|d\mu+\epsilon\mu(A)
\]
On the other hand by Cauchy-Schwartz inequality on $H$ and
by~\eqref{e:galerkin_energy_isometry}.
\[
\langle X^{(N_k)},U\rangle_{L^2}
=\int \Bigl\langle X^{(N_k)},\frac X{\|X\|}\Bigr\rangle_H\mathbbm 1_Ad\mu
\leq\int \|X^{(N_k)}\|\mathbbm 1_Ad\mu
\leq\int_A\|\chi\|d\mu
\]
Taking again the weak limit, we have $\mu(A)=0$ and by the arbitrarity
of $\epsilon$, we get that $\mu$-a.e.~$\|X\|\leq\|\chi\|$. This can be
improved by the continuity of components, which implies that the maps
$t\mapsto\sum_{k\leq n}|X_k(t)|^2$ are continuous and hence
$Q$-a.s.~bounded for all $t$ and all $n$ by $\|\chi\|^2$. Letting
$n\to\infty$ we conclude.
\end{proof}

The strong existence statement for the linear model becomes a weak
existence statement for the non-linear model, due to
Proposition~\ref{prop:linearsol_solvesnonlinear}.

\begin{corollary}\label{cor:existence_nonlinear}
Given an initial condition $x\in H$ and $T>0$, there exists at least
one Leray solution of the non-linear system~\eqref{e:itodiffmodel} in
the class $L^\infty(\Omega\times[0,T];H)$.
\end{corollary}

\section{Anomalous dissipation}\label{sec:6-dissipation}

In this section we want to prove that $\|X(t)\|$ goes to zero in some
sense. We consider the differential equation for the second
moments~\eqref{e:diff2ndmomentclosed} and study the continuous-time
Markov chain that has it as its forward Kolmogorov equation. The
following proposition gives an explicit connection between the two.

\begin{proposition}\label{prop:energy_is_prob_not_explosion}
Suppose $L_i$ is the identity matrix for all $i\in I$.
Let $x\in H$ and let $(Q,X)$ be the unique Leray solution of the
linear system~\eqref{e:itodifflinearmodel} with initial condition
$x$. Then there exists a continuous-time Markov chain
$(\xi_t)_{t\geq0}$ defined on a probability space $(S,\mathcal
S,\mathcal P)$ taking values in $\mathbb N$ and with q-matrix $\Pi$
defined by~\eqref{e:def_Pi}, such that for all $t\geq0$
\begin{align}\label{e:2ndmom_is_mclaw}
\mathbb E^Q\bigl[|X_n(t)|^2\bigr]
&=\|x\|^2\mathcal P(\xi_t=n)
,\qquad\forall n\geq1\\
\label{e:energy_is_bdr_prob}
\mathbb E^Q\bigl[\|X(t)\|^2\bigr]
&=\|x\|^2\mathcal P(\xi_t\in\mathbb N)
=\|x\|^2\mathcal P(\tau>t)
\end{align}
where 
\[
\tau:=\sup\{t:\xi\text{ has finitely many jumps in }[0,t)\}\in(0,\infty]
\]
is the so-called explosion time of the Markov chain.
\end{proposition}
\begin{proof}
Let $p_n^0:=|x_n|^2/\|x\|^2$ for all $n\geq1$.  It is standard to
formally construct a continuous-time Markov chain $\xi_t$ on
$(S,\mathcal S,\mathcal P)$ with initial distribution $p^0$ and rates
$\pi_{n,m}$ as defined in~\eqref{e:def_Pi}. Heuristically, the
process starts at a random position $\xi_0$ with $\mathcal
P(\xi_0=n)=p^0_n$. Then every time the process arrives in a position
$n$ it waits for an exponentially distributed random time with rate
$\pi_n=\sum_{m\neq n}\pi_{n,m}$ and then jumps to a new random
position different from $n$ cheosen with probabilities
$\genfrac(){}{}{\pi_{n,m}}{\pi_n}_{m\neq n}$. This defines $\xi_t$ up
to time $\tau$. At time $\tau$ we say that $\xi$ has reached the
boundary. (Sometimes this is done by adding one absorbing point
$\theta$ to the state space.)

For $n\geq1$, $t\geq0$, let $p_n(t):=\mathcal P(\xi_t=n)$. Then $p$ is
a non-negative solution of
\[
\begin{cases}
p'=p\Pi\\
p(0)=p^0
\end{cases}
\]
in $L^\infty([0,\infty);l^1)$. By
  Corollary~\ref{cor:2ndQmomentsolvesforward}, $u/\|x\|^2$ is another
  such solution, hence by the uniqueness result in
  Lemma~\ref{lem:uniqforwardsymm} we have
  proved~\eqref{e:2ndmom_is_mclaw} and by summing up
  also~\eqref{e:energy_is_bdr_prob}.
\end{proof}

The following is the main result for the anomalous dissipation of the
auxiliary linear system. The exponential decay of the expected value
of energy follows from the Markov property of the chain.

\begin{theorem}\label{thm:Q_avg_energy_exp_down}
Suppose $L_i$ is the identity matrix for all $i\in I$.  Let $x\in H$
and let $(Q,X)$ be the unique Leray solution of the linear
system~\eqref{e:itodifflinearmodel} with initial condition $x$.  Then
the quantity $\mathbb E^Q\bigl[\|X(t)\|^2\bigr]$ is strictly
decreasing in $t$. Moreover there exists a constant $\mu>0$ depending
only on the coefficients $k_{i,n}$ and a constant $C\geq\|x\|^2$
depending only on $k_{i,n}$'s and $x$, such that for all $t\geq0$
\[
\mathbb E^Q\bigl[\|X(t)\|^2\bigr]
\leq Ce^{-\frac{\sigma^2}\mu t}
\]
\end{theorem}
\begin{proof}
By Proposition~\ref{prop:energy_is_prob_not_explosion} we are given a
probability space $(S,\mathcal S,\mathcal P)$ and a continuous time
Markov chain $\xi$ on the positive integers, defined up to some
stopping time $\tau$ such that $\mathbb
E^Q[\|X(t)\|^2]=\|x\|^2\mathcal P(\tau>t)$, so we study the latter
probability.

Once we will prove the second statement, the fact that $\mathcal
P(\tau>t)$ is strictly decreasing in $t$ will follow from $\mathcal
P(\tau=\infty)<1$ by Chapman-Kolmogorv equation. (This is a standard
exercise on continuous-time Markov chains whose proof is not
difficult. See for example Lemma~13 in Barbato, Flandoli and
Morandin~\cite{BarFlaMor2011AAP}.)

We want to prove the exponential bound. Let $\zeta_k$ for
$k=0,1,2,\dots$ be the \emph{discrete} time Markov chain embedded in
$\xi$, meaning that $\zeta_k=\xi_t$ for $t$ between the $k$-th and the
$(k+1)$-th times of jump of $\xi$.

For $n\geq1$, let $V_n:=\sharp\{k\geq1:\zeta_k=n\}$ be the number of
times $\zeta_k$ visits $n$. The law of $V_n$, conditioned on
$V_n\neq0$ is geometrically distributed. Since the sum of a
geometrically distributed number of i.i.d.~exponential r.v.'s is
exponentially distributed, the total time $T_n$ spent by $\xi_t$ on
$n$, conditioned on ever reaching that site, is exponentially
distributed. For $n\geq1$ we define
\begin{gather*}
T_n:=\mathcal L\{t\geq0:\xi_t=n\}\\
\nu_n:=\mathbb E^{\mathcal P}[T_n|T_n>0]=E^{\mathcal P}[V_n|V_n>0]\pi_n^{-1}
\end{gather*}
so that $\tau=\sum_{n\geq1}T_n$ and for all $t\geq0$
\[
\mathcal P(T_n>t)
\leq\mathcal P(T_n>t|T_n>0)
=e^{-t/\nu_n}
\]

We claim that $E^{\mathcal P}[V_n|V_n>0]$ converges to some finite
limit as $n\to\infty$.

Then, since by Remark~\ref{rem:I_n_stable},
$\pi_n=\sigma^2\lambda^{2n}\sum_{i\in I_n}k_i^2=O(\lambda^{2n})$, we
have $\nu_n=O(\lambda^{-2n})$, so that the quantities
$\nu:=\sum_{n\geq1}\nu_n$ and $\Lambda:=-\sum_{n\geq1}\nu_n\log\nu_n$
are both finite. Define the sequence of numbers $(\theta_n)_{n\geq1}$
satisfying
\[
e^{-t\theta_n/\nu_n}=\nu_ne^{(\Lambda-t)/\nu}
,\qquad n\geq1
\]
and notice that
\[
\sum_{n\geq1}\theta_n
=\frac1t\sum_{n\geq1}\Bigl[-\nu_n\log\nu_n -(\Lambda-t)\frac{\nu_n}\nu\Bigr]
=1
\]
so we conclude that
\[
\mathcal P(\tau>t)
\leq\mathcal P\biggl(\bigcup_{n\geq1}\{T_n>\theta_nt\}\biggr)
\leq\sum_{n\geq1}\mathcal P(T_n>\theta_nt)
\leq\sum_{n\geq1}e^{-t\theta_n/\nu_n}
=\nu e^{(\Lambda-t)/\nu}
\]
This proves the theorem for
\begin{gather*}
C
=\|x\|^2\nu e^{\Lambda/\nu}
=\|x\|^2\exp\Bigl\{-\sum_{n\geq1}\frac{\nu_n}\nu\log\frac{\nu_n}\nu\Bigr\}
\geq\|x\|^2\\
\mu=\sigma^2\nu
=\sum_{n\geq1}\frac{E^{\mathcal P}[V_n|V_n>0]}{\lambda^{2n}\sum_{i\in I_n}k_i^2}
\end{gather*}
We check that these do not depend on $\sigma$. From the definitions of
$\nu_n$ and $\nu$, it is enough to show that the law of $\zeta_k$ does
not depend on $\sigma$.

The transition probabilities of $\zeta$ are given by $p_{n,n}=0$ and
$p_{n,m}:=\frac{\pi_{n,m}}{\pi_n}$ for $n\neq m$. Recall
from~\eqref{e:def_Pi} that
\begin{equation}\label{e:def_Pi_again}
\pi_{n,m}=\sigma^2\lambda^{2n}\!\!\!\sum_{\substack{i\in I_n\\r_i=m-n}}\!\!\!k_i^2
\qquad\text{and}\qquad
\pi_n=\sigma^2\lambda^{2n}\sum_{i\in I_n}k_i^2
\end{equation}
meaning in particular that $p_{n,m}$ does not depend on
$\sigma$.

\medskip
Finally, we must prove the claim.  

Consider $p_{n,n+r}$ and notice that again by~\eqref{e:def_Pi_again} and
Remark~\ref{rem:I_n_stable}, it does not depend on $n$, for $n\geq
n_0$. This means that, as long as $\zeta_k\geq n_0$, $\zeta$ behaves
like a random walk with increment distribution
\[
q_r
:=\frac{\sum_{\substack{i:r_i=r}}k_i^2}
  {\sum_{j\in I}k_j^2}
,\qquad r\in\mathbb Z.
\]
Let $\rho_k$, for $k=0,1,2,\dots$ be a random walk on $\mathbb Z$
defined on $(S,\mathcal S,\mathcal P)$ starting from $\rho_0=\zeta_0$,
with increment distribution $q$. Since
\[
\sum_{j\in I}k_j^2q_{-r}
=\sum_{\substack{i\in I\\r_i=-r}}k_i^2
=\sum_{\substack{i\in I\\r_{\tilde\imath}=r}}k_{\tilde\imath}^2\lambda^{2r_i}
=\sum_{j\in I}k_j^2q_r\lambda^{-2r}
\]
and $\lambda>1$, we have that $q_{-r}<q_r$ whenever $r>0$, so $\rho$
has a positive drift.

Now we forget for a moment the starting distribution of $\zeta$ and
consider only transition probabilities. Let $H=\{\rho_k\geq
n_0,\forall k\geq0\}$ and $K=\{\zeta_k\geq n_0,\forall k\geq0\}$ and
let $n\geq n_0$. Then
\begin{align}\label{e:H_K_same_prob}
\mathcal P(K|\zeta_0=n)
&=\mathcal P(H|\rho_0=n)\\
\label{e:return_same_prob_under_H_K}
\mathcal P(\zeta_k\neq n,\forall n\geq 1|\zeta_0=n,K)
&=\mathcal P(\rho_k\neq n,\forall n\geq 1|\rho_0=n,H)
\end{align}
Take the limit for $n\to\infty$. Since $\rho$ is a random walk with a
positive drift, then~\eqref{e:H_K_same_prob} converges to 1. Hence in
the limit we can drop $H,K$ from~\eqref{e:return_same_prob_under_H_K}
and conclude that
\[
\lim_{n\to\infty}\mathcal P(\zeta_k\neq n,\forall n\geq 1|\zeta_0=n)
=\lim_{n\to\infty}\mathcal P(\rho_k\neq n,\forall n\geq 1|\rho_0=n)>0
\]
This in particular means that $E^{\mathcal P}[V_n|V_n>0]$ converges to
some finite limit, which was the claim we had to prove.
\end{proof}

The statement of Theorem~\ref{thm:Q_avg_energy_exp_down} is about
expectations, but since the decay is at least exponential, it can be
refined to an almost sure convergence by virtue of Borel-Cantelli
Lemma. Proposition~\ref{prop:Q_energy_pathw_exp_down} below gives the
details.

\begin{lemma}\label{lem:Leray_s_t}
Under the same hypothesis of Theorem~\ref{thm:Q_avg_energy_exp_down},
let $s\geq0$, then $Q$-a.s.
\[
\sup_{t\geq s}\|X(t)\|\leq\|X(s)\|
\]
\end{lemma}
\begin{proof}
Let $\chi=X(s)$, and consider the restriction of $X$ to the time
interval $[s,\infty)$. Then by
  Theorems~\ref{thm:strong_existence_linear}
  and~\ref{thm:uniqueness_linear} and the ensuing remark, $X$ is the
  unique strong Leray solution and has the property that $Q$-almost
  surely $\|X(t)\|\leq\|\chi\|$.
\end{proof}

\begin{proposition}\label{prop:Q_energy_pathw_exp_down}
Under the same hypothesis of Theorem~\ref{thm:Q_avg_energy_exp_down},
the total energy of the solution goes to zero at least exponentially fast
pathwise under $Q$,
\[
\limsup_{t\to\infty}\frac1t\log\|X(t)\|^2\leq-\frac{\sigma^2}\mu
,\qquad Q\text{-a.s.}
\]
\end{proposition}
\begin{proof}
Let $\alpha>0$. By Theorem~\ref{thm:Q_avg_energy_exp_down} we can
bound the probabilities
\[
Q(\|X(n)\|^2>e^{-\alpha n})
\leq Ce^{-n\sigma^2/\mu}e^{\alpha n}
,\qquad n\geq0
\]
If we take $\alpha<\sigma^2/\mu$, by Borel-Cantelli lemma there exists a
r.v.~$M$, such that $Q$-a.s.~and for all $n\geq0$ we have
$\|X(n)\|^2\leq M e^{-\alpha n}$.

For $n=0,1,\dots$, apply Lemma~\ref{lem:Leray_s_t} with $s=n$ to get
that $Q$-a.s.
\[
\sup_{t\in[n,n+1)}\|X(t)\|^2
\leq M e^{-\alpha n}
\]
or
\[
\sup_{t\in[n,n+1)}\|X(t)\|^2e^{\alpha t-\alpha}
\leq\sup_{t\in[n,n+1)}\|X(t)\|^2e^{\alpha n}
\leq M 
\]
These are countably many propositions, so $Q$-a.s.~all of them are
true, yielding
\[
\sup_{t\geq0}\|X(t)\|^2e^{\alpha t}
\leq M e^\alpha
,\qquad Q\text{-a.s.}
\]
From here the thesis follows quickly by letting $\alpha\nearrow\sigma^2/\mu$
on the rational numbers.
\end{proof}

To translate the almost sure statement of the above proposition to the
initial non-linear problem, one should be able to prove the
equivalence of $P$ and $Q$ on $\mathcal F_\infty$ (see
Remark~\ref{rem:P_Q_equivalence}). The following proposition is the
key result to prove Novikov condition of Girsanov theorem for
$t=\infty$, which is the object of
Theorem~\ref{thm:expo_decr_nonlinear} below. A very similar statement
can be found in Barbato, Flandoli and Morandin~\cite{BarFlaMor2011AAP}
and the proof, which is almost the same, is given here for
completeness.

\begin{proposition}\label{prop:finite_exp_moment}
Under the same hypothesis of Theorem~\ref{thm:Q_avg_energy_exp_down},
let $\mu>0$ be the constant given there and let $\theta>0$. If
$\theta<\frac {\sigma^2}{\mu\|x\|^2}$, then
\[
\mathbb E^Q\bigl(e^{\theta\int_0^\infty\|X(t)\|^2dt}\bigr)<\infty
\]
\end{proposition}

\begin{proof}
Let $V:=\int_0^\infty\|X(t)\|^2dt$ and take any $v\geq0$. Let $u\geq0$
defined by 
\begin{equation}\label{e:u_def}
\|x\|^2Q(V>v)=Ce^{-u\sigma^2/\mu}
\end{equation}
where $\mu>0$ and $C\geq\|x\|^2$ are the constants given by
Theorem~\ref{thm:Q_avg_energy_exp_down}. Then
\begin{align*}
vQ(V>v)
&\leq\mathbb E^Q(V;V>v)
=\int_0^\infty\mathbb E^Q(\|X(t)\|^2;V>v)dt\\
&\leq\int_0^\infty\min\bigl(\|x\|^2Q(V>v);\mathbb E^Q(\|X(t)\|^2)\bigr)dt\\
&\leq\int_0^u\|x\|^2Q(V>v)dt+\int_u^\infty Ce^{-t\sigma^2/\mu}dt\\
&\leq u\|x\|^2Q(V>v)+\mu\sigma^{-2} Ce^{-u\sigma^2/\mu}\\
&= \|x\|^2Q(V>v)(u+\mu\sigma^{-2})
\end{align*}
where we used Leray property, Theorem~\ref{thm:Q_avg_energy_exp_down}
and twice equation~\eqref{e:u_def}.

If $Q(V>v)=0$ for some $v$ then $V$ is bounded and we are done.
Otherwise we get a lower bound on $u$ which put into~\eqref{e:u_def}
gives
\[
Q(V>v)
=\frac C{\|x\|^2}e^{-u\sigma^2/\mu}
\leq\frac {Ce^{-1}}{\|x\|^2}\exp\Bigl\{-\frac {\sigma^2}{\mu\|x\|^2}v\Bigr\}
\]
yielding $\mathbb E^Q(e^{\theta V})<\infty$ for all $\theta<\frac
{\sigma^2}{\mu\|x\|^2}$.
\end{proof}

\begin{theorem}\label{thm:expo_decr_nonlinear}
Under the same hypothesis of Theorem~\ref{thm:Q_avg_energy_exp_down},
let $\mu>0$ be the constant given there and let
$\rho:=\frac{\sqrt{\mu|I|}\|x\|}{2\sigma^2}$. If $\rho<1$, the total
energy of the solution goes to zero at least exponentially fast under
$P$, both pathwise and in average,
\begin{gather}\label{e:P_as_energy_exp_down}
\limsup_{t\to\infty}\frac1t\log\|X(t)\|^2\leq-\frac{\sigma^2}\mu
,\qquad P\text{-a.s.}\\
\label{e:P_avg_energy_exp_down}
\limsup_{t\to\infty}\frac1t\log\mathbb E^P\|X(t)\|^2\leq-\frac{\sigma^2}\mu(1-\rho)^2
\end{gather}
Moreover $P$ and $Q$ are equivalent on $\mathcal F_\infty$.
\end{theorem}
\begin{proof}
Novikov condition~\eqref{e:novikov_T} can be extended also to the
case $T=\infty$, and Proposition~\ref{prop:finite_exp_moment} applied
with $\theta=\frac{|I^*|}{2\sigma^2}=\frac{|I|}{4\sigma^2}$ shows that
it holds on $(\Omega, Q)$, so $P\ll Q$. The density is given
by~\eqref{e:QPdensity} with $t=\infty$; it is a.s.~positive, $P$ and
$Q$ are equivalent.  The first statement then follows by
Proposition~\ref{prop:Q_energy_pathw_exp_down}.

To prove~\eqref{e:P_avg_energy_exp_down} we follow Barbato, Flandoli
and Morandin~\cite{BarFlaMor2011AAP}. Fix $t>0$ and let
$f=\frac{dP_t}{dQ_t}=\exp\{\tilde Z_t-\frac12[\tilde Z,\tilde Z]_t\}$
(see~\eqref{e:QPdensity}). Let $p,q>1$ with $\frac1p+\frac1q=1$. Then
\[
\mathbb E^P\|X(t)\|^2
=\mathbb E^Q\bigl(f\|X(t)\|^2\bigr)
\leq\mathbb E^Q\bigl(f^p\bigr)^{1/p}\mathbb E^Q\bigl(\|X(t)\|^{2q}\bigr)^{1/q}
\]
We bound the first term by~\eqref{e:quadr_var_Z_bound_T} and Girsanov
theorem
\begin{align*}
\mathbb E^Q\bigl(f^p\bigr)
&=\mathbb E^Q\bigl(\exp\{p\tilde Z_t-\frac p2[\tilde Z,\tilde Z]_t\}\bigr)\\
&=\mathbb E^Q\bigl(\exp\{p\tilde Z_t-\frac12[p\tilde Z,p\tilde Z]_t+\frac{p(p-1)}2[\tilde Z,\tilde Z]_t\}\bigr)\\
&\leq\exp\{\frac{p(p-1)}2\frac{|I^*|\|x\|^2t}{\sigma^2}\}\mathbb E^Q\bigl(\exp\{p\tilde Z_t-\frac12[p\tilde Z,p\tilde Z]_t\}\bigr)\\
&=\exp\{\frac{p(p-1)|I|\|x\|^2t}{4\sigma^2}\}
\end{align*}
We bound the second term by Leray property and
Theorem~\ref{thm:Q_avg_energy_exp_down},
\[
\mathbb E^Q\bigl(\|X(t)\|^{2q}\bigr)
\leq\|x\|^{2q-2}\mathbb E^Q\bigl(\|X(t)\|^2\bigr)
\leq\|x\|^{2q-2}C e^{-\frac{\sigma^2}\mu t}
\]
Putting together the two bounds above and with some algebraic
manipulations we get that for all $p>1$
\[
\log\mathbb E^P\|X(t)\|^2
\leq\log\|x\|^2+\Bigl(1-\frac1p\Bigr)\Bigl(p\rho^2\frac{\sigma^2}\mu t-\frac{\sigma^2}\mu t+\log\frac C{\|x\|^2}\Bigr)
\]
This formula can be optimized on $p$. The RHS attains its minimum when
\[
p^2=\rho^{-2}(1-\frac\mu{\sigma^2t}\log\frac C{\|x\|^2})
\]
If $t$ is large enough and $\rho<1$, then this gives an acceptable
value $p>1$. By substituting this value of $p$ and letting
$t\to\infty$ we get the thesis.
\end{proof}

\section{Applications}\label{sec:7-applications}

In this section we apply our general model to two important shell
models of turbulence, namely the inviscid versions of GOY model
\begin{equation}\label{e:goy}
\frac d{dt}u_n
=ia\lambda_nu^*_{n+1}u^*_{n+2}+ib\lambda_{n-1}u^*_{n-1}u^*_{n+1}+ic\lambda_{n-2}u^*_{n-1}u^*_{n-2}
,\quad n\geq1
\end{equation}
and the inviscid version of Sabra model
\begin{equation}\label{e:sabra}
\frac d{dt}u_n
=ia\lambda_nu^*_{n+1}u_{n+2}+ib\lambda_{n-1}u^*_{n-1}u_{n+1}-ic\lambda_{n-2}u_{n-1}u_{n-2}
,\quad n\geq1
\end{equation}
In both models for $n\geq1$, $u_n$ are complex-valued functions,
$\lambda_n=\lambda^n$, $\lambda>1$, $a$, $b$, $c$ are real numbers
with $a+b+c=0$ and we set $\lambda_n=0$, $u_n=0$ for $n\leq0$ for
simplicity.

We may add multiplicative noise to both models to fall in two special
cases of our general model~\eqref{e:mainmodel}. This must be done
according to the initial requirements and it turns out that the proper
way to add noise is for the GOY
\begin{multline}\label{e:goy_stochastic}
du_n
=ia\lambda_nu^*_{n+1}u^*_{n+2}dt+ib\lambda_{n-1}u^*_{n-1}u^*_{n+1}dt+ic\lambda_{n-2}u^*_{n-1}u^*_{n-2}dt\\
+i\tilde\sigma\lambda_nu^*_{n+1}\strt dw_n
-i\tilde\sigma\lambda_{n-1}u^*_{n-1}\strt dw_{n-1}
\end{multline}
and for Sabra
\begin{multline}\label{e:sabra_stochastic}
du_n
=ia\lambda_nu^*_{n+1}u_{n+2}dt+ib\lambda_{n-1}u^*_{n-1}u_{n+1}dt-ic\lambda_{n-2}u_{n-1}u_{n-2}dt\\
+i\tilde\sigma_1\lambda_nu^*_{n+1}\strt dw_n
-i\tilde\sigma_1\lambda_{n-1}u^*_{n-1}\strt dw_{n-1}\\
+(i\tilde\sigma_2\lambda_nu^*_{n+1}\strt dw'_n)^*
-i\tilde\sigma_2\lambda_{n-1}u_{n-1}\strt dw'_{n-1}
\end{multline}
where $\tilde\sigma,\tilde\sigma_1,\tilde\sigma_2$ are positive
constants with $\frac{\tilde\sigma_1}{\tilde\sigma_2}=\frac{\lambda
  a}c$ and $(w_n)_{n\in\mathbb Z}$, $(w'_n)_{n\in\mathbb Z}$ are two
sequences of complex-valued Brownian motions which are all
independent.

\begin{definition}
Given an initial condition $u^0\in l^2(\mathbb C)$, a Leray solution
of the stochastic GOY system~\eqref{e:goy_stochastic} (respectively of
the stochastic Sabra system~\eqref{e:sabra_stochastic}) is a filtered
probability space $(\Omega, \mathcal F_\infty, (\mathcal
F_t)_{t\geq0}, P)$, along with an adapted sequence $(w_n)_{n\in\mathbb
  Z}$ (resp.~two adapted sequences $(w_n)_{n\in\mathbb Z}$ and
$(w'_n)_{n\in\mathbb Z}$ ) of independent complex-valued Brownian
motions, and a stochastic process $u$ such that
\begin{enumerate}[\hspace{2em}\it i.]
\item $u=(u_n)_{n\geq1}$ is a stochastic process on $(\Omega, \mathcal
  F_\infty, (\mathcal F_t)_{t\geq0}, P)$ taking values in $l^2(\mathbb
  C)$ with continuous adapted components;
\item With probability 1, for all $t\geq0$,
  $\|u(t)\|_{l^2}\leq\|u^0\|_{l^2}$.
\item the following integral equation
  (resp.~equation~\eqref{e:sabra_stoc_integ}) holds for all $n\geq1$
  and all $t\geq0$,
\begin{multline}\label{e:goy_stoc_integ}
u_n(t)
=u_n^0
+\int_0^tia\lambda_nu^*_{n+1}(s)u^*_{n+2}(s)ds
+\int_0^tib\lambda_{n-1}u^*_{n-1}(s)u^*_{n+1}(s)ds\\
+\int_0^tic\lambda_{n-2}u^*_{n-1}(s)u^*_{n-2}(s)ds
-\int_0^t\frac{\tilde\sigma^2}2(\lambda_n^2+\lambda_{n-1}^2)u_n(s)ds\\
+\int_0^ti\tilde\sigma\lambda_nu^*_{n+1}(s)dw_n(s)
-\int_0^ti\tilde\sigma\lambda_{n-1}u^*_{n-1}(s)dw_{n-1}(s)
\end{multline}
\end{enumerate}
\begin{multline}\label{e:sabra_stoc_integ}
u_n(t)
=u_n^0
+\int_0^tia\lambda_nu^*_{n+1}(s)u_{n+2}(s)ds
+\int_0^tib\lambda_{n-1}u^*_{n-1}(s)u_{n+1}(s)ds\\
-\int_0^tic\lambda_{n-2}u_{n-1}(s)u_{n-2}(s)ds
-\int_0^t\frac{\tilde\sigma_1^2+\tilde\sigma_2^2}2(\lambda_n^2+\lambda_{n-1}^2)u_n(s)ds\\
+\int_0^ti\tilde\sigma_1\lambda_nu^*_{n+1}(s)dw_n(s)
-\int_0^ti\tilde\sigma_1\lambda_{n-1}u^*_{n-1}(s)dw_{n-1}(s)\\
+\int_0^t(i\tilde\sigma_2\lambda_nu^*_{n+1}(s)dw'_n(s))^*
-\int_0^ti\tilde\sigma_2\lambda_{n-1}u_{n-1}(s)dw'_{n-1}(s)
\end{multline}
\end{definition}

\begin{theorem}\label{thm:stoc_goy}
Given an initial condition $u^0\in l^2(\mathbb C)$, there exists a
Leray solution $(P,u)$ of the stochastic GOY system which is unique in
law. Moreover for all $t>0$,
\[
P(\|u(t)\|_{l^2}<\|u^0\|_{l^2})>0
\]
and for all $\epsilon>0$ there exists $t>0$ such that
\[
P(\|u(t)\|_{l^2}<\epsilon)>0
\]
Finally, if $\|u^0\|_{l^2}$ is sufficiently small, then for
$t\to\infty$, $u(t)$ converges to zero at least exponentially fast
both almost surely and in $L^2(\Omega;l^2(\mathbb C))$.
\end{theorem}

\begin{proof}
All we need to do is rewrite this model in the formalism of our
general model~\eqref{e:mainmodel}. Take $d=2$, let $\phi:\mathbb
C\to\mathbb R^2$ be the obvious isomorphism and let
$X_n:=\phi(u_n):=(\re(u_n),\im(u_n))$ and $x_n=\phi(u^0)$. 

The first step is to define a bilinear operator $B$ on $\mathbb R^2$
corresponding to $(v,z)\mapsto iv^*z^*$ on $\mathbb C$. For
$\alpha,\beta,\gamma\in\{1,2\}$, let
\[
B^{\alpha,\beta,\gamma}=\begin{cases}
1/\sqrt2 & \alpha+\beta+\gamma=4\\
-1/\sqrt2 & \alpha+\beta+\gamma=6\\
0 & \text{otherwise}
\end{cases}
\]
so that it is easy to chack that for any $v,z\in\mathbb C$,
$\phi(iv^*z^*)=\sqrt2B(\phi(v),\phi(z))$ and that
$L=L^{\alpha,\beta}=\sum_{\gamma,\delta}B^{\alpha,\gamma,\delta}B^{\beta,\gamma,\delta}$
is the identity.

The second step is to choose the interactions corresponding to the GOY
local range coupling. Since there are three terms, at least two pair
of interactions are needed. Let $I=\{1,2,3,4\}$, $\tau=(1\;3)(2\;4)$,
$I^*=\{1,2\}$ and for $i\in I$ let $B_i=B$ and
\[
\begin{array}{c|ccc}
i & r_i & h_i & k_i \\
\hline
1 & \phantom-1 & \phantom-2 & \phantom-\sqrt2a \\
2 & -1 & -2 & \phantom-\sqrt2\lambda^{-2}c \\
3 & -1 & \phantom-1 & -\sqrt2\lambda^{-1}a \\
4 & \phantom-1 & -1 & -\sqrt2\lambda^{-1}c
\end{array}
\]
It is now easy to check that if we apply $\phi$ to the sum of the
first three integrals appearing in the RHS of
equation~\eqref{e:goy_stoc_integ}, we simply get
\begin{equation}
\sum_{i\in I}\int_0^tk_{i,n}B_i(X_{n+r_i}(s),X_{n+h_i}(s))ds
\end{equation}

Finally, let $\sigma:=\tilde\sigma/\sqrt{a^2+\lambda^{-2}c^2}$, let
$W=(W_{i,n})_{i\in I,n\in\mathbb Z}$ be a family of 2-dimensional
Brownian motions symmetric with respect to $\tau$, let
$(w_n)_{n\in\mathbb Z}$ be a sequence of independent complex-valued
Brownian motions and suppose that the following equation holds for all
$n$,
\begin{equation}\label{e:assumption_w_W}
w_n=
\frac{aW_{1,n+2}^1-\lambda^{-1}cW_{2,n-1}^1}{\sqrt{a^2+\lambda^{-2}c^2}}
-i\frac{aW_{1,n+2}^2-\lambda^{-1}cW_{2,n-1}^2}{\sqrt{a^2+\lambda^{-2}c^2}}
\end{equation}
Then
\begin{multline*}
\int_0^ti\tilde\sigma\lambda_nu^*_{n+1}(s)dw_n(s)\\
=\int_0^ti\sigma\lambda_nu^*_{n+1}adW_{1,n+2}^1
-\int_0^ti\sigma\lambda_nu^*_{n+1}\lambda^{-1}cdW_{2,n-1}^1\\
-\int_0^ti\sigma\lambda_nu^*_{n+1}iadW_{1,n+2}^2
+\int_0^ti\sigma\lambda_nu^*_{n+1}i\lambda^{-1}cdW_{2,n-1}^2\\
=\int_0^ti\sigma a\lambda_nu^*_{n+1}(dW_{1,n+2}^1-idW_{1,n+2}^2)
-\int_0^ti\sigma\lambda^{-1}c\lambda_nu^*_{n+1}(dW_{2,n-1}^1-idW_{2,n-1}^2)
\end{multline*}
so that we can compute $\phi$ applied to the two stochastic integrals
appearing in the RHS of equation~\eqref{e:goy_stoc_integ}: we obtain
\begin{multline*}
\phi\biggl(\int_0^ti\tilde\sigma\lambda_nu^*_{n+1}(s)dw_n(s)\biggr)\\
=\int_0^ti\sigma\sqrt2a\lambda_nB(X_{n+1},dW_{1,n+2})
-\int_0^ti\sigma\sqrt2\lambda^{-1}c\lambda_nB(X_{n+1},dW_{2,n-1})\\
=\sum_{i=1,4}\int_0^t\sigma k_{i,n}B_i(X_{n+r_i},W_{i,n+h_i})
\end{multline*}
and analogously
\[
\phi\biggl(-\int_0^ti\tilde\sigma\lambda_{n-1}u^*_{n-1}(s)dw_{n-1}(s)\biggr)
=\sum_{i=2,3}\int_0^t\sigma k_{i,n}B_i(X_{n+r_i},dW_{i,n+h_i})
\]
The last term is
\begin{multline*}
\phi\biggl(
-\int_0^t\frac{\tilde\sigma^2}2(\lambda_n^2+\lambda_{n-1}^2)u_n(s)ds
\biggr)
=-\int_0^t\frac{\sigma^2}2(a^2+\lambda^{-2}c^2)(1+\lambda^{-2})\lambda^{2n} X_n(s)ds\\
=-\sum_{i\in I}\int_0^t\frac{\sigma^2}2k_{i,n}^2X_n(s)ds
\end{multline*}
We have proved that, under the assumption
that~\eqref{e:assumption_w_W} holds, if we apply $\phi$ to
equation~\eqref{e:goy_stoc_integ}, we get
equation~\eqref{e:itointmodel}. But of course, given
$(W_{i,n}^j)_{i\in I, j=1,2}$, then~\eqref{e:assumption_w_W} may be
taken as a definition of $w_n$, so existence of a Leray solution
follows from Corollary~\ref{cor:existence_nonlinear}. On the other
hand, given $w$, let $\tilde w$ be onother sequence of independent
complex-valued Brownian motions, independent from $w$. Then
\[
\begin{pmatrix}
W_{1,n+2}\\
W_{2,n-1}
\end{pmatrix}
:=\begin{pmatrix}
W_{3,n+2}\\
W_{4,n-1}
\end{pmatrix}
:=\frac1{\sqrt{a^2+\lambda^{-2}c^2}}
\begin{pmatrix}
a & -\lambda^{-1}c\\
\lambda^{-1}c & a
\end{pmatrix}\begin{pmatrix}
\re(w_n) & \im(w_n) \\
\re(\tilde w_n) & \im(\tilde w_n) \\
\end{pmatrix}
\]
defines the family $W$ according to the requirements and
to~\eqref{e:assumption_w_W}, so uniqueness in law of the Leray
solution follows from Theorem~\ref{thm:uniqueness_nonlinear}.

To prove the two inequalities, remember that $\|X(t)\|=\|u(t)\|_{l^2}$
and apply Theorem~\ref{thm:Q_avg_energy_exp_down}. Fix $t>0$. Since
$\mathbb E^Q(\|X(t)\|)<\|x\|$, then $Q(\|X(t)\|<\|x\|)>0$, so the same
holds for $P$ which is equivalent to $Q$ on $\mathcal F_t$.

Fix $\epsilon>0$. Since $\mathbb E^Q(\|X(t)\|)\to0$ as $t\to\infty$,
then for $t$ large enough $Q(\|X(t)\|>\epsilon)<1$, so the same holds
for $P$ which is equivalent to $Q$ on $\mathcal F_t$.

Finally, to prove the last statement, we apply
Theorem~\ref{thm:expo_decr_nonlinear}. If $\|x\|=\|u^0\|_{l^2}$ is
small enough, then $\rho<1$, so by~\eqref{e:P_as_energy_exp_down} we
get that $P$-a.s.~for all $\epsilon>0$, for $t$ large
\[
\|u(t)\|_{l^2}\leq e^{-\frac12(\frac{\sigma^2}\mu-\epsilon)t}
\]
and by~\eqref{e:P_avg_energy_exp_down} we get that for all
$\epsilon>0$, for $t$ large
\[
\|u(t)\|_{L^2(\Omega;l^2(\mathbb C))}^2=\mathbb E^P\|u(t)\|_{l^2}^2\leq e^{-\left(\frac{\sigma^2}\mu(1-\rho)^2-\epsilon\right)t}\qedhere
\]
\end{proof}

\begin{theorem}\label{thm:stoc_sabra}
Given an initial condition $u^0\in l^2(\mathbb C)$, there exists a
Leray solution $(P,u)$ of the stochastic Sabra system which is unique
in law. Moreover for all $t>0$,
\[
P(\|u(t)\|_{l^2}<\|u^0\|_{l^2})>0
\]
and for all $\epsilon>0$ there exists $t>0$ such that
\[
P(\|u(t)\|_{l^2}<\epsilon)>0
\]
Finally, if $\|u^0\|_{l^2}$ is sufficiently small, then for
$t\to\infty$, $u(t)$ converges to zero at least exponentially fast
both almost surely and in $L^2(\Omega;l^2(\mathbb C))$.
\end{theorem}

\begin{proof}
We follow the same strategy as for Theorem~\ref{thm:stoc_goy}, so let
$d$, $\phi$, $X$, $x$, $I$, $\tau$, $r_i$, $h_i$ and $k_i$ be defined
like there. We need three new different bilinear operators on $\mathbb
R^2$ (below on the right) which represent the corresponding bilinear
operators on $\mathbb C$ associated to the interactions in the Sabra
model (below on the left)
\begin{align*}
(v,z)&\mapsto iv^*z &
B_1^{\alpha,\beta,\gamma}=B_3^{\alpha,\beta,\gamma}
&=\begin{cases}
0 & \alpha+\beta+\gamma\text{ odd}\\
-1/\sqrt2 & \alpha=1,\beta=1,\gamma=2\\
1/\sqrt2 & \text{otherwise}
\end{cases}\\
(v,z)&\mapsto-ivz &
B_2^{\alpha,\beta,\gamma}
&=\begin{cases}
0 & \alpha+\beta+\gamma\text{ odd}\\
-1/\sqrt2 & \alpha=2,\beta=1,\gamma=1\\
1/\sqrt2 & \text{otherwise}
\end{cases}\\
(v,z)&\mapsto ivz^* &
B_4^{\alpha,\beta,\gamma}
&=\begin{cases}
0 & \alpha+\beta+\gamma\text{ odd}\\
-1/\sqrt2 & \alpha=1,\beta=2,\gamma=1\\
1/\sqrt2 & \text{otherwise}
\end{cases}
\end{align*}
These $B_i$ satisfy~\eqref{e:alias_B} and the corresponding $L_i$'s are
the identity.

It is immediate to verify that if we apply $\phi$ to the sum of the
first three integrals appearing in the RHS of
equation~\eqref{e:sabra_stoc_integ}, we simply get
\begin{equation}
\sum_{i\in I}\int_0^tk_{i,n}B_i(X_{n+r_i}(s),X_{n+h_i}(s))ds
\end{equation}
Finally, let
$\sigma:=\tilde\sigma_1/a=\tilde\sigma_2/(\lambda^{-1}c)$, let
$(w_n)_{n\in\mathbb Z}$ and $(w'_n)_{n\in\mathbb Z}$ be two sequences
of independent complex-valued Brownian motions and let
$W=(W_{i,n})_{i\in I,n\in\mathbb Z}$ be a family of 2-dimensional
Brownian motions symmetric with respect to $\tau$ such that
$P$-a.s.~$W_{1,n}=W_{3,n}=\phi(w_{n-2})$ and
$W_{2,n}=W_{4,n}=\phi(w'_{n+1})$.

Then it also easy to verify that if we apply $\phi$ to the sum of the
four stochastic integrals appearing in the RHS of
equation~\eqref{e:sabra_stoc_integ}, we get
\begin{equation}
\sum_{i\in I}\int_0^t\sigma k_{i,n}B_i(X_{n+r_i}(s),dW_{i,n+h_i}(s))
\end{equation}
The last term is
\begin{multline*}
\phi\biggl(
-\int_0^t\frac{\tilde\sigma_1^2+\tilde\sigma_2^2}2(\lambda_n^2+\lambda_{n-1}^2)u_n(s)ds
\biggr)\\
=-\int_0^t\frac{\sigma^2}2(a^2+\lambda^{-2}c^2)(1+\lambda^{-2})\lambda^{2n} X_n(s)ds
=-\sum_{i\in I}\int_0^t\frac{\sigma^2}2k_{i,n}^2X_n(s)ds
\end{multline*}
We have proved that if we apply $\phi$ to
equation~\eqref{e:sabra_stoc_integ}, we get
equation~\eqref{e:itointmodel}. Then one concludes exactly like in
Theorem~\ref{thm:stoc_goy}.
\end{proof}

\begin{remark}
The smallness condition on $\|u^0\|_{l^2}$ can be made precise by
computing $\mu$ as in the proof of
Theorem~\ref{thm:Q_avg_energy_exp_down}. One has only to observe that
in this case the discrete-time embedded Markov chain $\zeta_k$ is a
simple random walk on the positive integers reflected in 1 and with
positive drift $\frac{\lambda^2-1}{\lambda^2+1}$ and do some
computations. We give only the result and notice that this choice of
$\mu$ is not believed to be optimal. For both the stochastic GOY and
Sabra models, the condition $\rho<1$ is equivalent to
\[
\|u^0\|_{l^2}<\sqrt2(\lambda-\lambda^{-1})\sqrt{a^2-\lambda^{-2}c^2}\,\sigma^2
\]
\end{remark}

\bibliographystyle{plain} \bibliography{dyadic}

\begin{thebibliography}{10}

\bibitem{Anderson}
William~J. Anderson.
\newblock {\em Continuous-time {M}arkov chains, an applications-oriented
  approach}.
\newblock Springer Series in Statistics: Probability and its Applications.
  Springer-Verlag, New York, 1991.

\bibitem{BarBarBesFla2006}
D.~Barbato, M.~Barsanti, H.~Bessaih, and F.~Flandoli.
\newblock Some rigorous results on a stochastic goy model.
\newblock {\em Journal of Statistical Physics}, 125(3):677--716, 2006.

\bibitem{BarBiaFlaMor2013}
D.~Barbato, L.A. Bianchi, F.~Flandoli, and F.~Morandin.
\newblock A dyadic model on a tree.
\newblock {\em arXiv preprint arXiv:1207.2846}, 2012.

\bibitem{BarMor2013nodea}
D.~Barbato and F.~Morandin.
\newblock Positive and non-positive solutions for an inviscid dyadic model:
  well-posedness and regularity.
\newblock {\em Nonlinear Differential Equations and Applications}, pages 1--19,
  2012.
\newblock Article in Press.

\bibitem{BarFlaMor2010CRAS}
David Barbato, Franco Flandoli, and Francesco Morandin.
\newblock A theorem of uniqueness for an inviscid dyadic model.
\newblock {\em C. R. Math. Acad. Sci. Paris}, 348(9-10):525--528, 2010.

\bibitem{BarFlaMor2010PAMS}
David Barbato, Franco Flandoli, and Francesco Morandin.
\newblock Uniqueness for a stochastic inviscid dyadic model.
\newblock {\em Proc. Amer. Math. Soc.}, 138(7):2607--2617, 2010.

\bibitem{BarFlaMor2011AAP}
David Barbato, Franco Flandoli, and Francesco Morandin.
\newblock Anomalous dissipation in a stochastic inviscid dyadic model.
\newblock {\em Annals of Applied Probability}, 21(6):2424--2446, 2011.

\bibitem{BarFlaMor2011TAMS}
David Barbato, Franco Flandoli, and Francesco Morandin.
\newblock Energy dissipation and self-similar solutions for an unforced
  inviscid dyadic model.
\newblock {\em Trans. Amer. Math. Soc.}, 363(4):1925--1946, 2011.

\bibitem{BarMorRom2011}
David Barbato, Francesco Morandin, and Marco Romito.
\newblock Smooth solutions for the dyadic model.
\newblock {\em Nonlinearity}, 24(11):3083, 2011.

\bibitem{BesFer2012}
H.~Bessaih and B.~Ferrario.
\newblock Invariant gibbs measures of the energy for shell models of
  turbulence: the inviscid and viscous cases.
\newblock {\em Nonlinearity}, 25(4):1075, 2012.

\bibitem{Bianchi2013}
L.A. Bianchi.
\newblock Uniqueness for an inviscid stochastic dyadic model on a tree.
\newblock {\em preprint}, 2012.

\bibitem{Biferale2003}
L.~Biferale.
\newblock Shell models of energy cascade in turbulence.
\newblock {\em Annu. Rev. Fluid Mech.}, 35:441--468, 2003.

\bibitem{CheFri09}
A.~Cheskidov and S.~Friedlander.
\newblock The vanishing viscosity limit for a dyadic model.
\newblock {\em Physica D: Nonlinear Phenomena}, 238(8):783--787, 2009.

\bibitem{CheFriPav2007}
Alexey Cheskidov, Susan Friedlander, and Nata{\v{s}}a Pavlovi{\'c}.
\newblock Inviscid dyadic model of turbulence: the fixed point and {O}nsager's
  conjecture.
\newblock {\em J. Math. Phys.}, 48(6):065503, 16, 2007.

\bibitem{CheFriPav2010}
Alexey Cheskidov, Susan Friedlander, and Nata{\v{s}}a Pavlovi{\'c}.
\newblock An inviscid dyadic model of turbulence: the global attractor.
\newblock {\em Discrete Contin. Dyn. Syst.}, 26(3):781--794, 2010.

\bibitem{ChuMil2009}
I.~Chueshov and A.~Millet.
\newblock Stochastic 2d hydrodynamical type systems: Well posedness and large
  deviations.
\newblock {\em Applied Mathematics \& Optimization}, 61(3):379--420, 2010.

\bibitem{ConLevTit2006}
P.~Constantin, B.~Levant, and E.S. Titi.
\newblock Analytic study of shell models of turbulence.
\newblock {\em Physica D: Nonlinear Phenomena}, 219(2):120--141, 2006.

\bibitem{ConLevTit2007}
P.~Constantin, B.~Levant, and E.S. Titi.
\newblock Regularity of inviscid shell models of turbulence.
\newblock {\em Physical Review E}, 75(1):016304, 2007.

\bibitem{DaPFlaPriRoe}
Giuseppe {Da Prato}, Franco {Flandoli}, Enrico {Priola}, and Michael
  {R{\"o}ckner}.
\newblock Strong uniqueness for stochastic evolution equations in hilbert
  spaces perturbed by a bounded measurable drift.
\newblock {\em ArXiv e-prints}, June 2012.

\bibitem{DesNov74}
V.~N. {Desnianskii} and E.~A. {Novikov}.
\newblock Simulation of cascade processes in turbulent flows.
\newblock {\em Prikladnaia Matematika i Mekhanika}, 38:507--513, 1974.

\bibitem{Feller57}
William Feller.
\newblock On boundaries and lateral conditions for the {K}olmogorov
  differential equations.
\newblock {\em Ann. of Math. (2)}, 65:527--570, 1957.

\bibitem{Gledzer1973}
EB~Gledzer.
\newblock System of hydrodynamic type admitting two quadratic integrals of
  motion.
\newblock In {\em Soviet Physics Doklady}, volume~18, page 216, 1973.

\bibitem{KatPav04}
Nets~Hawk Katz and Nata{\v{s}}a Pavlovi{\'c}.
\newblock Finite time blow-up for a dyadic model of the {E}uler equations.
\newblock {\em Trans. Amer. Math. Soc.}, 357(2):695--708 (electronic), 2005.

\bibitem{KisZla05}
Alexander Kiselev and Andrej Zlato{\v{s}}.
\newblock On discrete models of the {E}uler equation.
\newblock {\em Int. Math. Res. Not.}, (38):2315--2339, 2005.

\bibitem{KryRoz79}
N.~V. Krylov and B.~L. Rozovski{\u\i}.
\newblock Stochastic evolution equations.
\newblock In {\em Current problems in mathematics, {V}ol. 14 ({R}ussian)},
  pages 71--147, 256. Akad. Nauk SSSR, Vsesoyuz. Inst. Nauchn. i Tekhn.
  Informatsii, Moscow, 1979.

\bibitem{sabra1998}
V.S. L'vov, E.~Podivilov, A.~Pomyalov, I.~Procaccia, and D.~Vandembroucq.
\newblock Improved shell model of turbulence.
\newblock {\em Physical Review E}, 58(2):1811, 1998.

\bibitem{ManSriSun2009}
U.~Manna, SS~Sritharan, and P.~Sundar.
\newblock Large deviations for the stochastic shell model of turbulence.
\newblock {\em NoDEA: Nonlinear Differential Equations and Applications},
  16(4):493--521, 2009.

\bibitem{obukhov1971}
AM~Obukhov.
\newblock Turbulence in an atmosphere with a non-uniform temperature.
\newblock {\em Boundary-Layer Meteorology}, 2(1):7--29, 1971.

\bibitem{Pardoux75phd}
{\'E}.~Pardoux.
\newblock {\em Equations aux d{\'e}riv{\'e}es partielles stochastiques non
  lineaires monotones: Etude de solutions fortes de type Ito}.
\newblock PhD thesis, Universit{\'e} Paris Sud, 1975.

\bibitem{PreRoe}
Claudia Pr{\'e}v{\^o}t and Michael R{\"o}ckner.
\newblock {\em A concise course on stochastic partial differential equations},
  volume 1905 of {\em Lecture Notes in Mathematics}.
\newblock Springer, Berlin, 2007.

\bibitem{RevuzYor}
Daniel Revuz and Marc Yor.
\newblock {\em Continuous martingales and {B}rownian motion}, volume 293 of
  {\em Grundlehren der Mathematischen Wissenschaften [Fundamental Principles of
  Mathematical Sciences]}.
\newblock Springer-Verlag, Berlin, third edition, 1999.

\bibitem{OhkYam1987}
M.~Yamada and K.~Ohkitani.
\newblock Lyapunov spectrum of a chaotic model of three-dimensional turbulence.
\newblock {\em Physical Society of Japan, Journal}, 56:4210--4213, 1987.

\end{thebibliography}

\end{document}